\documentclass[11pt]{amsart}
\usepackage{amssymb}
\usepackage{verbatim}
\usepackage{amscd}
\usepackage{amsmath}

\theoremstyle{plain}
\newtheorem{theorem}{Theorem}[section]
\newtheorem{corollary}[theorem]{Corollary}
\newtheorem{lemma}[theorem]{Lemma}
\newtheorem{proposition}[theorem]{Proposition}
\newtheorem{definition}[theorem]{Definition}

\newtheorem{remark}[theorem]{Remark}

\numberwithin{equation}{section}

\def\boxit#1{\vbox{\hrule\hbox{\vrule\kern3pt
     \vbox{\kern3pt#1\kern3pt}\kern3pt\vrule}\hrule}}

\newfont{\msam}{msam10}            
\newfont{\msym}{msbm10 scaled\magstep1}            

\newfont{\gotic}{eufm10 scaled\magstep1}

\newcommand{\ra}{\rightarrow}
\newcommand{\lra}{\mbox{\Huge $\longrightarrow$}}

\newcommand{\kra}{\kern-7pt\rightarrow\kern-7pt}

\def\bbr{{\mathbb R}}

\def\bbn{{\mathbb N}}
\def\bbc{{\mathbb C}}

\def\ra{\rightarrow}

\def\lra{\longrightarrow}

\numberwithin{equation}{section}

\def\begtab{\begin{tabbing} WW\=23/02: \= point 1\kill}

\def\NIL1{{\mathcal H^{3}}}

\def\wt{\widetilde}

\def\n2{\mathfrak{N}_2}

\def\frakm{{\mathfrak m}}

\begin{document}

\title[ The holonomy on the principal $U(n)$ bundle over $D_{n,m}$ ]{Topological property of the holonomy displacement on the principal $U(n)$-bundle over $D_{n,m},$ related to complex surfaces }

\author{Taechang Byun}
\address{Department of Mathematics and Statistics, Sejong University,
Seoul 143-747, Korea}
\email{tcbyun@sejong.ac.kr}



\maketitle

\begin{abstract}
Consider $D_{n,m} = U(n,m)/\left(U(n) \times U(m)\right)$, the dual of the the Grassmannian manifold and the principal $U(n)$ bundle over $D_{n,m},$
$U(n)\rightarrow  U(n,m)/U(m) \stackrel{\pi} \rightarrow  D_{n,m}$.
Given a nontrivial $X \in M_{m \times n}(\mathbb{C}),$ consider a two dimensional subspace $\mathfrak{m}' \subset \mathfrak{m} \subset \mathfrak{u}(n,m), $
induced by $X, iX \in M_{m \times n}(\mathbb{C}),$
and a complete oriented surface $S,$
related to $(X,g)  \in M_{m \times n}(\mathbb{C}) \times U(n,m), $
in the base space $D_{n,m}$ with a complex structure from $\mathfrak{m}'.$
Let $c$ be a smooth, simple, closed, orientation-preserving curve on $S$
parametrized by $0\leq t\leq 1$, and $\hat{c}$ its horizontal lift on
the bundle
$U(n) \ra U(n,m)/U(m) \stackrel{\pi}\ra D_{n,m} $.
Then
the holonomy displacement is given by the right action of $e^\Psi$ for some 
$
  \Psi \in \text{Span}_{\bbr}\{i(X^*X)^k\}_{k=1}^{q} \subset \mathfrak{u}(n),
  \: q=\text{rk}X,
$
such that 
$$
\hat{c}(1) = \hat{c}(0) \cdot  e^\Psi
\text{\hskip24pt and\hskip12pt } 
\text{Tr}(\Psi)= 2i \, \text{Area}(c),
$$
where $\text{Area}(c)$ is the area of the region on the surface $S$
surrounded by $c,$ 
obtained from a special 2-form $\omega_{(X,g)}$ on $S,$  
called an area form $\omega_{(X,g)}$ related to $(X,g)$ on $S.$ 
\end{abstract}

\section{Introduction}

Gauss-Bonnet Theorem shows a kind of relation between Riemannian Geometry and Topology through two kinds of curvatures -Gaussian curvature and geodesic curvature-(, angles if needed) and Euler-characterstic. In this paper, we explain a similar phenomenon in some principal bundles through area and holonomy displacement.  

\medskip
In \cite{Pin}, Pinkall showed that the holonomy displacement of a simple closed curve in the base space on the Hopf bundle $S^1 \ra S^3 \ra S^2$ depends on the area of its interior. Byun and Choi \cite{BC} generalized this result to the principal $U(n)$-bundle over the Grassmannian manifold $G_{n,m}$ of complex $n$-planes in $\bbc^{n+m},$
$$U(n) \ra U(n+m)/U(m) \ra G_{n,m},$$
by introducing $U_{m,n}(\bbc),$ where
$$
  U_{m,n}(\bbc):= 
  \{ 
     X \in M_{m \times n}(\bbc) 
     \,\, | \,\, 
     X^{*} X = \lambda I_n \text{ for some } \lambda \in \bbc - \{0\}
  \}.
$$
Especially, the result related to a complex surface in $G_{n,m}$ can be summarized as follows:

\begin{theorem}\cite{BC}
Assume $U(k), \; k=1,2, \cdots,$ has a metric, related to the Killing-Cartan form, given by
$$
  \langle A,B \rangle 
   = \tfrac{1}{k} \text{Re}\big(\text{Tr}(A^{*}B)\big), 
      \qquad A,B \in \mathfrak{u}(k),
$$
Consider a bundle 
$U(n) \rightarrow U(n+m)/U(m) \stackrel{pr} \longrightarrow G_{n,m}, $
where $pr: U(n+m)/U(m) \ra G_{n,m}$ is a Riemannian submersion.
Given a nontrivial $X\in U_{m,n}(\mathbb{C})$, a two dimensional subspace $\mathfrak{m}' \in \mathfrak{u}(n+m)$, induced by $X \text{ and } iX ,$ gives rise to a complete totally geodesic surface $S$ with a complex structure in the base space $G_{n,m}.$ 
And  
if $c:[0,1]\ra S$ is a piecewise smooth, simple, closed curve on $S$
and if $\tilde{c}$ is its horizontal lift,
then the holonomy displacement along $c$,
$$\wt c(1)= \wt c(0) \cdot V(c),$$
is given by the right action of
$
  V(c)=e^{i \theta} I_n  \in U(n),
$
where $A(c)$ is the induced area of the region, surrounded
by $c,$ on the surface $S,$   from the metric on $G_{n,m},$  
and $\theta= 2 \cdot \tfrac{n+m}{2n} A(c) .$
\end{theorem}
If the metric is changed into
$$
  \langle A,B \rangle 
   = \tfrac{1}{2} \text{Re}\big(\text{Tr}(A^{*}B)\big), 
      \qquad A,B \in \mathfrak{u}(k),
$$
then the the result will be done into $\theta= 2 \cdot \tfrac{1}{n} A(c),$
and from $e^{i\theta}I_n = e^{i\theta I_n}, $  it can be read as
$$
   i\theta I_n \in \mathfrak{u}(n) 
   \quad \text{and} \quad
   \text{Tr}(i \theta I_n) = i n \theta = 2i \,A(c),  
$$
where $A(c)$ is the changed area induced from the changed metric on $G_{n,m}.$

\bigskip
\vspace{0.3cm}
  On the other hand, Choi and Lee \cite{CL} thought of the dual version of Pinkall's result over $\bbc H^m:$ 
\begin{theorem}\cite{CL} \label{pull-back}
Let $S^1 \ra S^{2m,1} \ra \bbc H^m$ be the natural fibration. Let $S$ be a complete totally geodesic surface in $\bbc H^m,$ and $\xi_S$ be the pullback bundle over $S.$ Let $c$ be a piecewise smooth, simple closed curve on $S.$ Then the holonomy displacement along $c$ is given by
$$ V(c)= e^{\frac{1}{2}A(c)i} \text{ or } e^{0i} \: \in S^1,$$
where 
$
 A(c)
   \footnote{
         In \cite{CL}, 
         the bundle $S^1 \ra S^{2m,1} \ra \bbc H^m$ is being studied through 
         the one 
         $U(1) \ra U(1,m)/U(m) \ra U(1,m)/(U(1)\times U(m))$ with 
         $ U(1,m)/U(m)\cong S^{2m,1}$  
         and $U(1,m)/(U(1)\times U(m)) \cong\bbc H^m,$ 
         and the unmentioned diffeomorphism 
         $f:\bbc H^m \lra U(1,m)/(U(1)\times U(m))$ 
         is a conformal map satisfying $|f_* v| = \frac{1}{2} |v|.$  
   }
$ 
is the area of the region on the surface $S$ surrounded by $c,$ depending on whether $S$ is a complex submanifold or not.
\end{theorem}

\bigskip
Consider $D_{n,m}= U(n,m)/\big(U(n) \times U(m)\big).$

\medskip
We generalize these results in the principal $U(n)$-bundle
$$U(n) \ra U(n,m)/U(m) \stackrel{\pi}{\rightarrow} D_{n,m}$$ 
over $D_{n,m}$ up to $M_{m \times n}(\bbc)$ 
for general positive integers $n,m \in \bbn,$ 
not only up to $U_{m,n}(\bbc),$ as follows: consider a left invariant metric on $U(n,m)$, related to the Killing-Cartan form, given by
\begin{align} \label{metric}
  \langle A,B \rangle 
   &= \tfrac{1}{2} \text{Re}\big(\text{Tr}(A^{*}B)\big), 
      \qquad A,B \in \mathfrak{u}(n,m),
\end{align}   
and  
the induced metric on $D_{n,m},$ which makes the natural projection 
$$\tilde{\pi}: U(n,m)  \lra D_{n,m}$$
a Riemannian submersion
and induces another Riemannain submersion
$$\pi: U(n,m)/U(m) {\longrightarrow} D_{n,m}.$$
Given a nontrivial $X \in M_{m \times n}(\bbc),$ let
$$
  \widehat{X} =
  \left(
        \begin{array}{cccc}
          O_n & X^* \\
           X  & O_m \\
        \end{array}
  \right)
  \quad  \text{and} \quad
  \widehat{iX} =
  \left(
        \begin{array}{cccc}
          O_n & -iX^* \\
          iX  & O_m \\
        \end{array}
  \right).
$$
Then for $W= \tfrac{1}{|\widehat{X}|} X,$ 
for $a,b \in \bbr$ and for $z=a+bi \in \bbc$
$$
  \widehat{zW}
  = 
  \widehat{\tfrac{a+bi}{ \big|\widehat{X}\big|} X}
  = 
  \tfrac{a}{ \big|\widehat{X}\big|} \widehat{X} +
  \tfrac{b}{ \big|\widehat{iX}\big|} \widehat{iX}
  = a \widehat{W} + b \widehat{iW}.
$$
Consider a complete surface 
$\tilde{S}=\{e^{\widehat{zW}} | z \in \bbc\}$ in $U(n,m).$
Then the map $z \mapsto e^{\widehat{zW}}: \bbc \ra \tilde{S}$ is a bijection.
For $g \in U(n,m),$
define \emph{a complex surface $S$ related to $(X,g)$} in $D_{n,m}$ by
$S=\tilde{\pi}(g\tilde{S}),$
which has a complex structure induced from a 2-dimensional subspace 
$
 \frakm' = \emph{Span}_{\bbr} \{\widehat{X},\widehat{iX}\} 
         = \{\widehat{zW} | z= x+iy \text{ for } x,y\in \bbr\} 
                 \subset \frakm  \subset \mathfrak{u}(n,m), 
$ 
where 
$\frakm$ is the orthogonal complement of 
$\mathfrak{u}(n) + \mathfrak{u}(m).$

\medskip

In case of  
$$U(1) \ra U(1,m)/U(m) \ra U(1,m)/(U(1) \times U(m))$$
and in case of
$$U(n) \ra U(n,1)/U(1) \ra U(n,1)/(U(n) \times U(1)),$$
the Lie algebra generated by $\{\widehat{X}, \widehat{iX}\}$ from
$X \in M_{m \times n}(\bbc),$ either $m=1$ or $n=1,$ inducing a complex surface $S$ in the base space, 
produces a 3-dimensional Lie subgroup $\widehat{G}$ of $U(n,m),$ 
which is isomorphic to $SU(1,1),$ 
and a bundle structure isomorphic to
$$S(U(1)\times U(1)) \lra SU(1,1) \lra SU(1,1)/S(U(1)\times U(1)).$$
Then, under the notaion $A(c)$ of the induced area of the region surrounded
by $c$ on the surface $S$   from the metric on the base space,  
this enables us to guess the pull-back bundle
and the holonomy displacement 
     \footnote{
          As explained in the footnote to $A(c)$ in Theorem \ref{pull-back},
          $V(c)$ in \cite{CL} is given by $e^{\tfrac{1}{2}A(c)i}.$
      }
$$
  V(c)=e^{2A(c)i}
$$ 
along a curve $c$ in the base space \cite{CL},
which also enables us to guess its induced holonomy displacement in the original bundle from either $X^*X \in \bbr$ or $XX^* \in \bbr,$
from Lemma \ref{sing} and from Proposition  \ref{prop}. 
We deal with the latter one in Section  \ref{special2}.

But for general positive integers $n,m \in \bbn$ and 
for a general nontrivial $X \in M_{m \times n}(\bbc),$ 
the Lie algebra generated by $\{\widehat{X}, \widehat{iX}\}$ is not 3-dimensional.
Furthermore, the holonomy displacement depends not only on $X$  
but also on some 2-form of the complex surface $S$ related to it too heavily.
From now on, we consider two kinds of 2-forms on $S,$ the first one is related to $(X,g) \in M_{m\times n}(\bbc) \times U(n,m),$ defined in Definition \ref{def}, and the other one is induced from the metric on $S$ obtained from the metric on $D_{n,m},$ mentioned in Proposition \ref{prop}.

\medskip

To begin with, given $X \in M_{m \times n}(\bbc),$ 
think of 
$W= \tfrac{1}{|\widehat{X}|} X$ 
and
$\tilde{S}=\{e^{\widehat{zW}}|z \in \bbc\},$
which is one to one correspondent to  $S_0 := \tilde{\pi}(\tilde{S}).$
Refer to Remark \ref{bijection}.
And  for any $\tilde{g} \in U(n,m),$ let
$\mathbb{L}_{\tilde{g}}: D_{n,m} \ra D_{n,m}$ be the action of $\tilde{g},$ induced from the left multiplication of $\tilde{g}$ on $U(n,m),$ which is an isometry from 
$\mathbb{L}_{\tilde{g}} \circ\tilde{\pi}= \tilde{\pi} \circ L_{\tilde{g}}.$

\begin{definition} \label{def}
Think of a bundle $U(n)\times U(m) \ra U(n,m)\stackrel{\tilde{\pi}}\lra D_{n,m}.$
Given  $(X,g) \in M_{m \times n}(\bbc) \times U(n,m),$ 
consider $W= \tfrac{1}{|\widehat{X}|} X$ and
define a 2-form $\omega_{(X,g)}$ on a complex surface $S$ related to $(X,g)$ in $D_{n,m},$ called 
\emph{an area form $\omega_{(X,g)}$ related to} $(X,g)$ \emph{on} $S,$ by
$$
  \omega_{(X,g)}(\tilde{\pi}_*x, \tilde{\pi}_*y) = \emph{det}
        \left(
              \begin{array}{cccc}
                 \langle {L_{{g_1}^{-1}}}_*x, \widehat{W} \rangle 
                 &\langle {L_{{g_1}^{-1}}}_*y, \widehat{W} \rangle \\
                 \langle {L_{{g_1}^{-1}}}_*x, \widehat{iW} \rangle 
                 &\langle {L_{{g_1}^{-1}}}_*y, \widehat{iW} \rangle \\
              \end{array}
        \right)
$$
under the identification of the tangent space of $U(n,m)$ at the identity and its  Lie algebra $\mathfrak{u}(n,m),$
where $g_1 \in g\tilde{S}=\{g e^{z\widehat{W}} | z\in \bbc\}$ and both $x$ and $y$ are tangent to $g\tilde{S}$ at $g_1.$
\end{definition}

\begin{proposition} \label{prop}
Think of a bundle $U(n)\times U(m) \ra U(n,m)\stackrel{\tilde{\pi}}\lra D_{n,m}$
such that $\tilde{\pi}$ is a Riemannian submersion.
Given a nontrivial $X \in  M_{m \times n}(\bbc),$ 
consider 
a complete, complex surface 
$\tilde{S} = \{e^{\widehat{zW}} \big| z \in \bbc \}$ in $U(n,m)$
for $W= \tfrac{1}{|\widehat{X}|} X$ and
for 
$
  \widehat{W}
  = 
  \widehat{\tfrac{1}{ \big|\widehat{X}\big|} X}
  = 
  \tfrac{1}{ \big|\widehat{X}\big|} \widehat{X}.
$  
\\
\indent
\emph{(i)}
For a complex surface $S_0 = \tilde{\pi}(\tilde{S})$ 
related to $(X,e)$ in $D_{n,m},$ 
let $\omega_0$ be the 
area form $\omega_{(X,e)}$ related to $(X,e)$ on $S_0,$ where $e$ is the identity of $U(n,m).$
Consider
a coordinate system $(r,\theta)$ on $S_0$ induced from 
$$S_0 = \{e^{\widehat{r e^{i \theta}W}} \big| z=re^{i\theta} \in \bbc\}.$$
Then,  for the differntial $d=d_{S_0}$ on $S_0,$
$$
  \omega_0 \,=\, \omega_{(X,e)} 
           \,=\, d\Big(
                       \sum_{j=1}^{n}  \tfrac{1}{2}\sinh^2{(\sigma_j r)} \,
                        d\theta
                  \Big),                  
$$
where $\sigma_1 \ge \cdots \ge \sigma_n$ are the square roots of decreasingly ordered non-negative eigenvalues of $W^*W$ 
with $\sigma_1^2 + \cdots + \sigma_n^2 =1.$
(Refer to \emph{Lemma \ref{sing}} for more information on the eigenvalues.) 
\\
\indent
\emph{(ii)}
Given a complex surface $S$ in $D_{n,m},$ 
related to $(X,g) \in M_{m \times n}(\bbc) \times U(n,m),$
the area form $\omega_{(X,g)}$ related to $(X,g)$ on $S$ is given by
$$\omega_{(X,g)}= {\mathbb{L}_{g^{-1}}}^*\omega_0.$$
\\
\indent
\emph{(iii)}
The induced area form $\omega$ on $S$ from the metric on $D_{n,m}$ is given by
$$\omega = \mathbb{L}_{g^{-1}}^* \omega_1$$
and
\begin{align*}
  \omega_1 
         &= \frac{1}{2}\,
            \sqrt{\sum_{j=1}^n \sinh^2{(2\sigma_jr)}} \,
            dr\wedge d\theta,
\end{align*}
where $\omega_1$ is the induced area form  on $S_0$ from the metric on $D_{n,m}.$
\\
\indent
Especially, if the positive eigenvalues of $W^*W$ consist of a single real number 
with allowing the duplication, that is, 
$\sigma_1 = \cdots = \sigma_q >0 = \sigma_{q+1} = \cdots = \sigma_n,$
then 
$$
  \omega_1 = \omega_0
         = d\big(\tfrac{q}{2} \,\sinh^2\big(\tfrac{r}{\sqrt{q}}\big)d\theta\big)
$$
and $q \in \bbn$ is the algebraic multiplicity of the positive eigenvalue $\sigma_1^2$ of $W^*W.$
\end{proposition}

\medskip
\begin{theorem} \label{thm}
Consider 
a bundle
$U(n) \ra U(n,m)/U(m) \stackrel{\pi}\ra D_{n,m} $
such that $\pi$ is a Riemannian submersion.
Given a complex surface $S$ in $D_{n,m},$ 
related to $(X,g) \in  M_{m \times n}(\bbc) \times U(n,m),$
let $c$ be a smooth, simple, closed, oreintation-preserving curve on 
$S,$ 
parametrized by $0\leq t\leq 1$, and $\hat{c}$ its horizontal lift.
Then
the holonomy displacement $\hat{c}(1)= \hat{c}(0) \cdot  e^\Psi$ is given by the right action of $e^\Psi$ for some 
$
  \Psi \in \emph{Span}_{\bbr}\{i(X^*X)^k\}_{k=1}^{q} \subset \mathfrak{u}(n),
  \: q =\emph{rk}X,
$
such that
$$
\emph{Tr}(\Psi)= 2i \, \emph{Area}(c),
$$
where $\emph{rk}X$ is the rank of $X$ and $\emph{Area}(c)$ is the area of the region on the surface $S$ surrounded by $c$ 
with respect to the area form $\omega_{(X,g)}$ related to $(X,g)$ on $S.$
\\
\indent
(For more information on how to find a  concrete 
$\Psi \in \mathfrak{u}(n)$ not only for $n=1$ but also for $n>1,$ refer to \emph{Remark} \ref{remark} and \ref{remark_soln}.)

\end{theorem}

\bigskip
\begin{remark}
One metric structure on $D_{n,m}$ induced from a Riemannian submersion $\tilde{\pi}: U(n,m) \ra D_{n,m}$ and the other one induced from a Riemannian submersion $\pi: U(n,m)/U(m) \ra D_{n,m}$, in fact, are same. See \emph{Section \ref{pre}}. 
\end{remark}

\bigskip
\begin{corollary} \label{corollary}
In addition to the hypothesis of Theorem \ref{thm}, assume that $X^*X$ has a single positive eigenvalue with algebraic multiplicity $q.$ Then, 
\begin{align*} 
  \Psi 
  &= A \;
     \emph{diag}[
                 \underbrace{
                             \tfrac{2i}{q}\emph{Area}(c),\cdots,
                             \tfrac{2i}{q}\emph{Area}(c),
                            }_{q \emph{ times}}
                 \underbrace{0,\cdots,0}_{(n-q) \emph{ times}}
                ]  \;
     A^* 
  \\             
  &= A \;
     \emph{diag}[
                 \underbrace{
                             \tfrac{2i}{q}A(c),\cdots,
                             \tfrac{2i}{q}A(c),
                            }_{q \emph{ times}}
                 \underbrace{0,\cdots,0}_{(n-q) \emph{ times}}
                ]  \;
     A^* 
\end{align*}
for some $A \in U(n),$ where $A(c)$ is the  area with respect to the induced metric on $S$ from the metric   on $D_{n,m},$ \emph{i.e.,} the area with respect to the induced area form $\omega$ on $S$ from the metric   on $D_{n,m}.$ 
\\
\indent
Especially, if $q=n,$ then $X \in U_{m,n}(\bbc), $  $S$ is totally geodesic and
$e^{\Psi}= e^{i \theta}I_n$ with 
$\theta = \frac{2}{n} \emph{Area}(c)= \frac{2}{n} A(c)$ 
\end{corollary}
\begin{proof}
From hypothesis, for $W = \tfrac{1}{|\widehat{X}|}X,$ $W^*W$ also has a single positive eigenvalue with algebraic multiplicity $q.$ Call its positive square root $\sigma.$ Then, under the notation of Lemma \ref{sing}, we get, for some $A \in U(n),$
\begin{align*}
  W^*W
  &= A \Sigma^* \Sigma A^* \\
  &= A \;
     \text{diag}[
                 \underbrace{\sigma^2,\cdots,\sigma^2,}_{q \text{ times}}
                 \underbrace{0,\cdots,0}_{(n-q) \text{ times}}
                ]  \;
     A^* 
  \\
  &= A \;
     \text{diag}[
                 \underbrace{
                             \tfrac{1}{q},\cdots,\tfrac{1}{q},
                            }_{q \text{ times}}
                 \underbrace{0,\cdots,0}_{(n-q) \text{ times}}
                ]  \;
     A^* 
\end{align*}
from Proposition \ref{prop}. Then, Theorem \ref{thm} and Proposition \ref{prop} say that
$$\Psi \in \text{Span}\{X^*X\} = \text{Span}\{W^*W\}$$
with
$$\text{Tr}(\Psi) = 2i \text{Area}(c)= 2i A(c),$$
which proves the former part of Corollary. 
\\
\indent
For the latter one, assume $q=n.$ Then the similar arguments as those for
$W^*W = A \Sigma^* \Sigma A^* $ before show that 
\begin{align} \label{eqn_norm}
  X^*X =\tilde{A}(\lambda I_n)\tilde{A} = \lambda I_n \in U_{m,n}(\bbc)
\end{align}
for some $\lambda >0$ and for some $\tilde{A} \in U(n).$ 
And $q=n$ implies
$$
  \Psi 
  = A \Big(\tfrac{2i}{n} \text{Area}(c) I_n \Big) A^*
  = \tfrac{2i}{n} \text{Area}(c) I_n
  = i \theta I_n
$$
for $\theta =\tfrac{2}{n} \text{Area}(c) $ 
and for some $A \in U(n),$
which gives
$$e^{\Psi} = e^{i \theta I_n} = e^{i \theta} I_n.$$
Furthermore, note that
$$
  W= \tfrac{1}{\sqrt{n \lambda}} X 
  \quad \text{and} \quad
  W^*W = \tfrac{1}{n} I_n
$$
from equation (\ref{eqn_norm}).
Then for
\begin{align*}
  \widehat{V} 
  \!=\!
  \left(
        \begin{array}{cccc}
           iW^*W & O\\
           O & -iWW^* \\
        \end{array}
  \right)\!
  \!=\!
  \left(
        \begin{array}{cccc}
           \tfrac{i}{n} I_n & O\\
           O & -iWW^* \\
        \end{array}
  \right)\!,
\end{align*}
a set $\{\widehat{X}, \widehat{iX}, \widehat{V}\},$ or
$$
  \{
    \tfrac{1}{\sqrt{\lambda}}\widehat{X}, 
    \tfrac{1}{\sqrt{\lambda}}\widehat{iX}, 
    n \widehat{V}
  \}
  =
  \{\sqrt{n}\,\widehat{W}, \sqrt{n}\,\widehat{iW}, n\widehat{V}\}
$$ 
generates a 3-dimensional Lie algebra with
$$
  [\sqrt{n}\,\widehat{W},\sqrt{n}\,\widehat{iW}] = 2 n\widehat{V}, \quad
  [n\widehat{V},\sqrt{n}\,\widehat{W}] = -2 \sqrt{n}\,\widehat{iW} 
  \text{ and } 
  [n\widehat{V},\sqrt{n}\,\widehat{iW}] = 2 \sqrt{n}\,\widehat{W}.
$$
Since each action of $g \in U(n,m)$, $\mathbb{L}_g : D_{n,m} \ra D_{n,m}$ is an isometry from 
$\mathbb{L}_g \circ \tilde{\pi} = \tilde{\pi} \circ L_g$
and $S = \mathbb{L}_g(S_0),$ Proposition \ref{affine} says that $S$ is totally geodesic.
\end{proof}

\bigskip
\begin{remark}
For general $q=1, \cdots, n,$ we can show that $S$ in 
\emph{Corollary \ref{corollary}} is totally geodesic by using  singular value decomposition: under the notation of \emph{Lemma \ref{sing}}, we get, for some $B \in U(m),$
\begin{align*}
  WW^*
  &= B \Sigma \Sigma^* B^* \\
  &= B \;
     \emph{diag}[
                 \underbrace{\sigma^2,\cdots,\sigma^2,}_{q \text{ times}}
                 \underbrace{0,\cdots,0}_{(m-q) \text{ times}}
                ]  \;
     B^* 
  \\
  &= B \;
     \emph{diag}[
                 \underbrace{
                             \tfrac{1}{q},\cdots,\tfrac{1}{q},
                            }_{q \text{ times}}
                 \underbrace{0,\cdots,0}_{(m-q) \text{ times}}
                ]  \;
     B^* .
\end{align*}
If we let
$$
  \Omega =   
        \left(
              \begin{array}{cccc}
                A & O \\
                O  &  B \\
              \end{array}
        \right)
$$
and 
$$
  \widehat{V} =   
        \left(
              \begin{array}{cccc}
                iW^*W & O \\
                O  &  -iWW^* \\
              \end{array}
        \right),
$$
then by using singular value decomposition, we get
a set $\{\widehat{X}, \widehat{iX}, \widehat{V}\},$ or
$$
  \{\sqrt{q}\,\widehat{W}, \sqrt{q}\,\widehat{iW}, q\widehat{V}\}
$$ 
generates a 3-dimensional Lie algebra with
$$
  [\sqrt{q}\,\widehat{W},\sqrt{q}\,\widehat{iW}] = 2 q\widehat{V}, \quad
  [q\widehat{V},\sqrt{q}\,\widehat{W}] = -2 \sqrt{q}\,\widehat{iW} 
  \text{ and } 
  [q\widehat{V},\sqrt{q}\,\widehat{iW}] = 2 \sqrt{q}\,\widehat{W},
$$
so \emph{Proposition \ref{affine}} says that $S$ is totally geodesic.
\end{remark}

\bigskip

\begin{remark} \label{remark}
$\Psi$ in \emph{Theorem \ref{thm}} is obtained by a solution of a system of first order linear differential equations:
for $t \in [0,1],$
consider a curve $z(t) = r(t)e^{i\theta(t)}$ in $\bbc$ and another one $\tilde{\gamma}(t)$ in $\tilde{S} \subset U(n,m)$ such that 
$\tilde{\gamma}(t)=e^{\widehat{z(t)W}}$ and that 
$\tilde{\pi} \circ L_g \circ \tilde{\gamma} = c,$ 
where $W= \tfrac{1}{\big|\widehat{X}\big|}X$
and $L_g: U(n,m)\ra U(n,m)$ is an isometry by the left multiplication of 
$g \in U(n,m).$
For a curve $\Psi(t) \in \mathfrak{u}(n),$ given by 
$$
  \Psi(t) =   
        \left(
              \begin{array}{cccc}
                i \sum_{k=1}^{q} \phi_{k}(t)(W^*W)^{k} & O \\
                O  &  O_{m} \\
              \end{array}
        \right)
  \in \mathfrak{u}(n),
$$
let $\hat{c}(t) = \big(g\,\tilde{\gamma}(t) e^{\Psi(t)}\big)\, U(m)$
be a horizontal lift of $c(t),$
where $\phi_k(t)$ is to be determined for each $k=1, \cdots q.$
Then $\Psi$ is given by $\Psi = \Psi(1) - \Psi(0).$
Furthermore, for some $A \in U(n)$ and $B \in U(m)$ 
from \emph{Lemma \ref{sing}}, we will get
$$
  \Psi(t)
  =
  \Omega \:
  \emph{diag}\Big[       
                   i \sum_{k=1}^{q} \sigma_1^{2k}\phi_k(t), 
                   \: \cdots, \: 
                   i \sum_{k=1}^{q} \sigma_{q}^{2k}\phi_k(t),
                   \underbrace{0, \: \cdots, 0}_{(n-q) \emph{ times}}, 
                   \underbrace{0, \: \cdots, 0}_{m \emph{ times}}
              \Big] \:
  \Omega^*,
$$
where
\begin{align*}
  \Omega
  =   
        \left(
              \begin{array}{cccc}
                A & O \\
                O & B \\
              \end{array}
        \right), 
\end{align*}
and  a system of a first order linear differential equations
\begin{align} \label{horizontal_condition}
  \sum_{k=1}^{q}\sigma_j^{2k}\phi_k'(t)
  = \theta'(t) \sinh^2(\sigma_jr(t)),
  \qquad j=1, \cdots, q.
\end{align}

Refer to \emph{Section \ref{sec}} to see why this system comes. 
And for the concrete expression of $\Psi(t),$ refer to \emph{Remark} \ref{remark_soln}.
\end{remark}

\medskip
\begin{remark} \label{remark_soln}
For  simplicity, we can regard $\Psi(t)$ in \emph{Remark \ref{remark}} as
$$
  \Psi(t) =  i \sum_{k=1}^{q} \phi_{k}(t)(W^*W)^{k} 
  \in \mathfrak{u}(n).
$$
Then we get
$$
  \Psi(t)
  =
  A \:
  \emph{diag}\Big[       
                   i \sum_{k=1}^{q} \sigma_1^{2k}\phi_k(t), 
                   \: \cdots, \: 
                   i \sum_{k=1}^{q} \sigma_{q}^{2k}\phi_k(t),
                   \underbrace{0, \: \cdots, 0}_{(n-q) \emph{ times}}
              \Big] \:
  A^*.
$$

If 
$
  \{ \sigma_1, \cdots, \sigma_q\} = \{ \sigma_{j_1}, \cdots, \sigma_{j_p}\}
$
with $\sigma_{j_1} > \cdots >\sigma_{j_p},$
then
for uniquely determined $(n \times n)$-matrices $I_{j_l}$'s from the equation
$
\emph{diag}[\sigma_1^2, \cdots, \sigma_q^2, 0, \cdots, 0]
=\sum_{l=1}^p {\sigma_{j_l}^2 I_{j_l}},
$
we have
\begin{align*}
W^*W
&= A \,
   \emph{diag}[
               \sigma_1^2, \cdots, \sigma_q^2, 
               \underbrace{0, \cdots, 0}_{(n-q) \emph{ times}} 
              ] \,
   A^*  \\
&= A \big( \sum_{l=1}^p {\sigma_{j_l}^2 I_{j_l}} \big) A^*  \\
&= \sum_{l=1}^p {\sigma_{j_l}^2 A I_{j_l} A^*}.
\end{align*}
Note
\begin{align*}
(W^*W)^k 
&= A \,
   \emph{diag}[
               \sigma_1^{2k}, \cdots, \sigma_q^{2k}, 
               \underbrace{0, \cdots, 0}_{(n-q) \emph{ times}} 
              ] \,
   A^*  \\
&= A \big( \sum_{l=1}^p {\sigma_{j_l}^{2k} I_{j_l}} \big) A^*  \\
&= \sum_{l=1}^p {\sigma_{j_l}^{2k} A I_{j_l} A^*}
\end{align*}
for $k \in \{1, \cdots, q\},$ which says that
$\emph{Span}_{\bbr}\{i(X^*X)^k\}_{k=1}^q$ in 
\emph{Theorem \ref{thm}} has a dimension $p$ and 
a basis $\{ A I_{j_l} A^* \}_{l=1}^p.$ 
For $k=1, \cdots, p,$
this system can be regarded as
\begin{align*}
   \left(
         \begin{array}{cccc}
           \sigma_{j_1}^{2\cdot 1}  &  \cdots & \sigma_{j_p}^{2\cdot1} \\
           \vdots                   &  \ddots & \vdots \\
           \sigma_{j_1}^{2p}  &  \cdots & \sigma_{j_p}^{2p} \\
         \end{array}
   \right)
   \left(
         \begin{array}{cccc}
           AI_{j_1}A^* \\
           \vdots  \\
           AI_{j_p}A^* \\
         \end{array}
   \right)
   =
   \left(
         \begin{array}{cccc}
           (W^*W)^1\\
           \vdots  \\
           (W^*W)^p\\
         \end{array}
   \right).
\end{align*}
Then for the invertible $(p \times p)$-matrix D on the left hand side,
\begin{align*}
   D=
   \left(
         \begin{array}{cccc}
           \sigma_{j_1}^{2\cdot 1}  &  \cdots & \sigma_{j_p}^{2\cdot1} \\
           \vdots                   &  \ddots & \vdots \\
           \sigma_{j_{1}}^{2p}  &  \cdots & \sigma_{j_p}^{2p} \\
         \end{array}
   \right),
\end{align*}
we get 
\begin{align*}
   \left(
         \begin{array}{cccc}
           AI_{j_1}A^* \\
           \vdots  \\
           AI_{j_p}A^* \\
         \end{array}
   \right)
   =
   D^{-1}
   \left(
         \begin{array}{cccc}
           (W^*W)^1\\
           \vdots  \\
           (W^*W)^{p}\\
         \end{array}
   \right),
\end{align*}
which says that $\{(W^*W)^k\}_{k=1}^p$ is also a basis.
From this basis, the curve $\Psi(t)$ may be reconstructed by
\begin{align*}
\Psi(t)
&= \sum_{k=1}^p {i \psi_k(t) (W^*W)^k} \\
&= A\,
  \Big( 
      \sum_{l=1}^p \,
         \big(
             \,\sum_{k=1}^p  {i \psi_k(t) \sigma_{j_l}^{2k}\,
         \big) 
         \, I_{j_l}} 
   \Big)\, 
   A^*,
\end{align*}
which change a system of first order differential equations \emph{(\ref{horizontal_condition})} for constructing a horizontal curve condition into another one
$$
  \sum_{k=1}^{p}\sigma_{j_l}^{2k}\psi_k'(t)
  = \theta'(t) \sinh^2(\sigma_{j_l}r(t)),
  \qquad l=1, \cdots, p,
$$
whose intial value is given by $\Psi(0)$ satisfying
$\hat{c}(0) = \big(g\,\tilde{\gamma}(0) e^{\Psi(0)}\big)\, U(m).$
This system can be rewritten as
\begin{align*}
   \left(
         \begin{array}{cccc}
           \sigma_{j_1}^{2\cdot 1}  &  \cdots & \sigma_{j_1}^{2p} \\
           \vdots                   &  \ddots & \vdots \\
           \sigma_{j_p}^{2\cdot 1}  &  \cdots & \sigma_{j_p}^{2p} \\
         \end{array}
   \right)
   \left(
         \begin{array}{cccc}
           \psi_1'(t) \\
           \vdots  \\
           \psi_p'(t) \\
         \end{array}
   \right)
   =
   \left(
         \begin{array}{cccc}
           \theta'(t) \sinh^2(\sigma_{j_1}r(t))\\
           \vdots  \\
           \theta'(t) \sinh^2(\sigma_{j_p}r(t))\\
         \end{array}
   \right).
\end{align*}
Then for the invertible $(p \times p)$-matrix $C$ on the left hand side,
\begin{align*}
   C=
   \left(
         \begin{array}{cccc}
           \sigma_{j_1}^{2\cdot 1}  &  \cdots & \sigma_{j_1}^{2p} \\
           \vdots                   &  \ddots & \vdots \\
           \sigma_{j_{p}}^{2\cdot 1}  &  \cdots & \sigma_{j_p}^{2p} \\
         \end{array}
   \right)
   =D^T,
\end{align*}
we get a solution
\begin{align*}
   \left(
         \begin{array}{cccc}
           \psi_1(t) \\
           \vdots  \\
           \psi_p(t) \\
         \end{array}
   \right)
   =
   \int{
         C^{-1}
   \left(
         \begin{array}{cccc}
           \theta'(t) \sinh^2(\sigma_{1}r(t))\\
           \vdots  \\
           \theta'(t) \sinh^2(\sigma_{p}r(t))\\
         \end{array}
   \right)
   dt.
      }
\end{align*}
Especially, if $X \in U_{m,n}(\bbc),$ then 
$p=1,$  $\sigma_1 = \cdots = \sigma_n,$ 
$$
  \emph{Span}_{\bbr}\{i(X^*X)^k\}_{k=1}^q
  =\{i\mu I_{n\times n}\,|\, \mu \in \bbr\},
$$
and
$$
\psi_1(t) 
= \int {\tfrac{1}{\sigma_1^2} \,\theta'(t) \sinh^2(\sigma_1 r(t)) dt},
$$
so
$\Psi(t)$ is given by
\begin{align*}
  \Psi(t)
  &=
  A \,(i \sigma_1^2 \psi_1(t) I_{n \times n})\, A^*
  \\
  &=
  i \sigma_1^2 \psi_1(t) I_{n \times n}
  \\
  &=
  \Psi(0) +
  i\Big(\int_0^t {\theta'(u) \sinh^2(\sigma_1 r(u)) du}\Big) I_{n\times n}.
\end{align*}
\end{remark}

\bigskip
\begin{remark} \label{bijection}
The singular value decomposition of $e^{\widehat{zW}},$
consisting of $\Omega, \Gamma_n(z), \Gamma_m(z)$ and $\Lambda(z),$
in \emph{Section \ref{sec_prop}} shows that 
$z \mapsto e^{\widehat{zW}}: \bbc \ra \tilde{S}$ is a bijection. Furthermore, it also shows that the restriction of $\tilde{\pi}$ on $\tilde{S},$ 
$$\tilde{\pi}: \tilde{S} \ra S_0$$
is injective(, so bijective): to show it, assume that 
$\tilde{\pi}(e^{\widehat{z_1W}}) = \tilde{\pi}(e^{\widehat{z_2W}})$
for $z_j=r_j e^{i \theta_j}, \, j=1,2,$ with $r_j \geq 0.$
Then, we get 
$$e^{\widehat{-z_1W}} e^{\widehat{z_2W}} \in U(n) \times U(m)$$
and the calculation through their singular value decomposition gives
\begin{align*}
  e^{\widehat{-z_1W}} e^{\widehat{z_2W}} 
  &=    \Omega
        \left(
              \begin{array}{cccc}
                \Gamma_n(z_1) & \Lambda(z_1)^* \\
                \Lambda(z_1)  & \Gamma_m(z_1) \\
              \end{array}
        \right)^{-1} 
        \left(
              \begin{array}{cccc}
                \Gamma_n(z_2) & \Lambda(z_2)^* \\
                \Lambda(z_2)  & \Gamma_m(z_2) \\
              \end{array}
        \right) 
        \Omega^*
  \\      
  &=    \Omega
        \left(
              \begin{array}{cccc}
                \Gamma_n(z_1) & -\Lambda(z_1)^* \\
                -\Lambda(z_1)  & \Gamma_m(z_1) \\
              \end{array}
        \right) 
        \left(
              \begin{array}{cccc}
                \Gamma_n(z_2) & \Lambda(z_2)^* \\
                \Lambda(z_2)  & \Gamma_m(z_2) \\
              \end{array}
        \right) 
        \Omega^*,
\end{align*}
and so, from  $\Omega \in U(n) \times U(m),$
\begin{align*}
        \left(
              \begin{array}{cccc}
                \Gamma_n(z_1) & -\Lambda(z_1)^* \\
                -\Lambda(z_1)  & \Gamma_m(z_1) \\
              \end{array}
        \right) 
        \left(
              \begin{array}{cccc}
                \Gamma_n(z_2) & \Lambda(z_2)^* \\
                \Lambda(z_2)  & \Gamma_m(z_2) \\
              \end{array}
        \right) 
        \in U(n) \times U(m),  
\end{align*}
whose $(1, n+1)$-element is
\begin{tiny}
\begin{align*}
  0
  &= 
   \Big(
       \cosh(\sigma_1r_1)\, 
       \underbrace{0\cdots 0}_{(n-1) \emph{ times}} \:\: 
       -\!e^{-i\theta_1}\sinh(\sigma_1r_1) \, 
       \underbrace{0\cdots 0}_{(m-1) \emph{ times}}
   \Big)
   \left(
         \begin{array}{cccc}
           e^{-i\theta_2}\sinh(\sigma_1r_2) \\
           0 \\
           \vdots  \\
           0 \\
           \cosh(\sigma_1r_2) \\
           0\\
           \vdots  \\
           0 \\
         \end{array}
   \right)
   \\
   &= e^{-i\theta_2} \cosh(\sigma_1r_1) \sinh(\sigma_1r_2)
      - e^{-i\theta_1} \cosh(\sigma_1r_2) \sinh(\sigma_1r_1)
   \\
   &= \big(
         \cos{\theta_2} \cosh(\sigma_1r_1) \sinh(\sigma_1r_2)
          - \cos{\theta_1} \cosh(\sigma_1r_2) \sinh(\sigma_1r_1)
      \big)
      \\
      & \hspace{0.5cm}
      -i   
      \big(
         \sin{\theta_2} \cosh(\sigma_1r_1) \sinh(\sigma_1r_2)
          - \sin{\theta_1} \cosh(\sigma_1r_2) \sinh(\sigma_1r_1)
      \big).
\end{align*}
\end{tiny}
Then 
\begin{align} 
     \label{inj_1}
     \cos{\theta_2} \cosh(\sigma_1r_1) \sinh(\sigma_1r_2)
      &= \cos{\theta_1} \cosh(\sigma_1r_2) \sinh(\sigma_1r_1),
      \\
      \label{inj_2}
      \sin{\theta_2} \cosh(\sigma_1r_1) \sinh(\sigma_1r_2)
      &= \sin{\theta_1} \cosh(\sigma_1r_2) \sinh(\sigma_1r_1),
\end{align}
so from $\cos^2\theta + \sin^2\theta =1,$
$$
\cosh^2(\sigma_1r_1) \sinh^2(\sigma_1r_2)
= \cosh^2(\sigma_1r_2) \sinh^2(\sigma_1r_1),
$$
which means either 
\begin{align*}
0
&=
\cosh(\sigma_1r_1) \sinh(\sigma_1r_2) - \cosh(\sigma_1r_2) \sinh(\sigma_1r_1)
\\
&= \sinh(\sigma_1(r_2-r_1))
\end{align*}
or
\begin{align*}
0
&=
\cosh(\sigma_1r_1) \sinh(\sigma_1r_2) + \cosh(\sigma_1r_2) \sinh(\sigma_1r_1)
\\
&= \sinh(\sigma_1(r_2+r_1)).
\end{align*}
Thus we get $r_1 = r_2 \geq 0.$
If both of them equals to $0,$ then $z_1=z_2=0.$
If $r_1=r_2 >0,$ then the equations (\ref{inj_1}) and (\ref{inj_2}) say that
$$
  \cos{\theta_1}=\cos{\theta_2} 
  \quad \text{and} \quad
  \sin{\theta_1}=\sin{\theta_2}, 
$$
which gives  $z_1 = r_1 e^{i \theta_1} = r_2 e^{i \theta_2} = z_2.$
\end{remark}

\bigskip

\section{ Preliminaries } \label{pre}
Given a submersion $pr: (M, g_M) \ra (B, g_B),$ 
the \mbox{\emph{vertical distribution}} $\mathcal{V}$ and the \mbox{\emph{horizontal distribution}} $\mathcal{H}=\mathcal{V}^{\perp}$ are defined to be the kernel of $pr_*$ and its orthogonal complement, respectively. And $pr$ is said to be \mbox{\emph{Riemannian}}  if $|pr_*x| = |x|$ for all $x \in \mathcal{H}.$ \cite{GW}

\medskip

Let $K$ be a subgroup of the isometry group of a Riemannian manifold $M,$ and suppose that all orbits have the same type, that is, any two are equivaiantly diffeormorphic. Then there exits a differentiable structure on $M/K$ with a Riemannian metric for each the natural projection $\pi: M \ra M/K$ is a Riemannian submersion. \cite{GW}

\medskip
Given a Lie group $G$ with a left-invariant metric 
and given a subgroup $H$ with the right multiplication $R_h$ an isometry for each $h \in H,$ the space of left cosets $G/H$ can be endowed with the metric for which the canonical projection $\pi: G \ra G/H$ is a Riemannain submersion. If we denote by $\mathbb{L}_g: G/H \ra G/H$ the action of $g \in G,$ 
then $\mathbb{L}_g \circ \pi = \pi \circ L_g$ and so $\mathbb{L}_g$ is an isometry of $G/H.$ \cite{GW}

\bigskip
Recall that
\begin{align*}
  G=U(n,m) 
  &= 
  \{ \Phi \in GL_{n+m}(\bbc) | 
         \Phi^* \: \Lambda^{n} _{m} \: \Phi \: =  \: \Lambda^{n} _{m}
  \}    \\
  &= 
  \{ \Phi \in GL_{n+m}(\bbc) \, \big| \, 
         F(\Phi v, \Phi w ) = F(v,w), \, v,w \in \bbc ^{n+m}
  \}
\end{align*}
and that
$D_{n,m} := U(n,m)/\left(U(n) \times U(m)\right)$ can be regarded as the set of $n$-dimensional subspaces $V$ of $\bbc^{n+m}$ such that $F(v,v) \le 0$ for every $v \in V$ \cite{KN}, 
where 
$F : \bbc ^{n+m} \rightarrow \bbc$ is an Hermitian form defined by
\begin{align*}
  F(v,w) & = v^* \, \Lambda^n _{m} \, w \\
         & = -\sum_{k=1}^{n} \bar{v}_k w_k + \sum_{s=n+1}^{n+m}\bar{v}_s w_s
\end{align*}         
for column vectors $v, w \in \bbc^{n+m},$  
$$
  \Lambda^{n} _{m} =
  \left(
        \begin{array}{cccc}
          -I_n & O_{n \times m} \\
          O_{m \times n} & I_m \\
        \end{array}
  \right).
$$

\bigskip

\bigskip
Consider the following canonical decomposition of the Lie algebra $\mathfrak{u}(n,m)$ of $G=U(n,m)$:
$$\mathfrak{u}(n,m) = \mathfrak{h} + \mathfrak{m},$$
where
$$
  \mathfrak{h}
  =\mathfrak{u}(n)+ \mathfrak{u}(m)
  =
  \left\{
        \left(
              \begin{array}{cccc}
                A & O_{n \times m} \\
                O_{m \times n} & B \\
              \end{array}
        \right)
        \:  : \:
        A \in \mathfrak{u}(n),\:  B \in \mathfrak{u}(m)  
  \right\}
$$
and
$$
  \mathfrak{m}
  =
  \left\{
        \hat{X} :=   
        \left(
              \begin{array}{cccc}
                O_{n} & X^{*} \\
                X     & O_{m}   \\
              \end{array}
        \right)
        \: : \:
        X \in M_{m \times n}(\bbc)  
  \right\}.
$$

Since the right multiplication $R_h,$ $h \in U(n) \times U(m),$ is an isometry with respect the left invariant metric given by the equaion (\ref{metric}),
there are two kinds of principal bundles 
$$U(m) \rightarrow U(n,m) \stackrel{\hat{\pi}} \lra U(n,m)/U(m) $$
and
$$U(n)\times U(m) \rightarrow U(n,m) \stackrel{\tilde{\pi}} \lra D_{n,m} $$
such that both $\hat{\pi}$ and $\tilde{\pi}$ are Riemannian submersions.
Note that each action of $g \in U(n,m)$ on $U(n,m)/U(m)$ and on $D_{n,m},$
denoted by  
$$\widehat{L}_g: U(n,m)/U(m) \lra U(n,m)/U(m),$$
and 
$$\mathbb{L}_g: D_{n,m} \lra D_{n,m},$$
respectively,
is an isometry from 
$$
  \widehat{L}_g \circ \hat{\pi} = \hat{\pi} \circ L_g
  \quad \text{and} \quad
  \mathbb{L}_g \circ \tilde{\pi} = \tilde{\pi} \circ L_g.
$$

For each $h_1 \in U(n),$ consider the right action 
$$\widehat{R}_{h_1}: U(n,m)/U(m) \ra U(n,m)/U(m)$$ 
given by 
$$ \widehat{R}_{h_1} \big( g U(m) \big) = \big(gU(m)\big)\cdot h_1:= (gh) U(m),$$
where 
$
  h =
  \left(
        \begin{array}{cccc}
               h_1       & O_{n \times m} \\
          O_{m \times n} & I_m \\
        \end{array}
  \right).
$
This action is well-defined from
$$
  \left(
        \begin{array}{cccc}
               h_1       & O \\
          O & I_m \\
        \end{array}
  \right)
  \left(
        \begin{array}{cccc}
               I_n       & O \\
          O & h_2 \\
        \end{array}
  \right)
  =
  \left(
        \begin{array}{cccc}
               I_n       & O \\
          O & h_2\\
        \end{array}
  \right)
  \left(
        \begin{array}{cccc}
               h_1       & O \\
          O & I_m \\
        \end{array}
  \right).
$$
for any $h_2 \in U(m).$
By abusing of notations, write $R_h = R_{h_1}. $
Then the equation 
$\widehat{R}_{h_1} \circ \hat{\pi} = \hat{\pi} \circ R_{h_1}$ implies that $\widehat{R}_{h_1}\!\!: U(n,m)/U(m) \ra U(n,m)/U(m)$ is an isometry 
since $R_{h_1}\!\!: U(n,m) \ra U(n,m)$ is an isometry preserving each fiber of the bundle 
$\hat{\pi}:U(n,m) \ra U(n,m)/U(m).$ More concretely, for any horizontal vector $x$ with respect to the Riemannian submersion 
$\hat{\pi}:U(n,m) \ra U(n,m)/U(m),$ 
$R_{h_1 *}\, x$ is a horizontal vector  since $R_{h_1}$  is an isometry preserving the fibers of the bundle $\hat{\pi}:U(n,m) \ra U(n,m)/U(m),$ so
\begin{align*}
  |\widehat{R}_{h_1 *}\, \hat{\pi}_* x|
  &= | \hat{\pi}_* R_{h_1 *}\, x | \\
  &= | R_{h_1 *}\, x | \:\: 
  \\
  &= |x| \\
  &= |\hat{\pi}_* x|.
\end{align*}

\bigskip

Consider another bundle
$$U(n) \ra U(n,m)/U(m) \stackrel{\pi} \lra D_{n,m}$$
and a metric structure on $D_{n,m},$
by regarding  
$D_{n,m}$ as the space of orbits of $U(n,m)/U(m)$ obtained from the action of a subgroup 
$\{\widehat{R}_{h_1} | h_1 \in U(n)\}$ of the isometry group of $U(n,m)/U(m),$
such that its projection $\pi$ is a Riemannian submersion.
In fact, for each $g \in U(n,m),$
$$
  g \big(U(n) \times U(m)\big)
  \:=\: \bigcup\{\widehat{R}_{h_1}\big(g U(m)\big) \;\big|\; h_1 \in U(n)\},
$$
which enables us to identify $D_{n,m}$ with the space of orbits through
$$\pi(gU(m)) = g \big(U(n) \times U(m)\big),$$
and then we get $\pi \circ \hat{\pi} = \tilde{\pi}.$
\bigskip

For each $g \in U(n,m)$ and for each $h_1\! \in U(n),$  it is obvious that both $\widehat{L}_g$ and $\widehat{R}_{h_1}$ preserve the fibers of the bundle $\pi: U(n,m)/U(m) \ra D_{n,m}.$
In fact, $\pi \circ \widehat{R}_{h_1}  =\pi$ by definition of $\pi$ and 
$\pi \circ \widehat{L}_g = \mathbb{L}_g \circ \pi$ from
$$
  (\pi \circ \widehat{L}_g) \circ \hat{\pi}
  = (\pi \circ \hat{\pi}) \circ L_g
  = \tilde{\pi} \circ L_g
  = \mathbb{L}_g \circ \tilde{\pi}
  = (\mathbb{L}_g \circ \pi) \circ \hat{\pi}.
$$

\bigskip

Note that we have given two metric structures on $D_{n,m}.$ 
In other words, we think of two Riemannian manifolds:
the first one is
the Riemannian manifold 
$(D_{n,m}, \langle \cdot, \cdot \rangle_1)$
with $\tilde{\pi}: U(n,m) \ra (D_{n,m}, \langle \cdot, \cdot \rangle_1)$ a Riemannian submersion
and
the other one is the Riemannian manifold 
$(D_{n,m}, \langle \cdot, \cdot \rangle_2)$
with $\pi: U(n,m)/U(m) \ra (D_{n,m}, \langle \cdot, \cdot \rangle_2)$ a Riemannian submersion.
But, 
$$
  \tilde{\pi}=\pi \circ \hat{\pi} 
  \qquad \text{and} \qquad 
  \hat{\pi} \circ L_{g^{-1}}= \widehat{L}_{g^{-1}}\circ\hat{\pi}, 
  \:\: \forall g \in U(n,m)
$$ 
say that
these two metric structures are same, that is,
$$|\tilde{\pi}_* x|_1 = |\tilde{\pi}_* x|_2$$
for any horizontal vector $x$ with respect to the Riemannian submersion $\tilde{\pi}:U(n,m) \ra (D_{n,m}, \langle \cdot, \cdot \rangle_1).$ 
To show it, assume 
$x \in T_gU(n,m), \,g \in U(n,m).$ The identification of the tangent space of $U(n,m)$ at the Identity and $\mathfrak{u}(n,m)$ gives
$L_{g^{-1} *} \,x \perp \big(\mathfrak{u}(n)\!+\!\mathfrak{u}(m)\big)$
and so $L_{g^{-1} *} \,x \perp \mathfrak{u}(m),$
which means that $L_{g^{-1} *} \,x$ is horizontal with respect to 
$\hat{\pi}: U(n,m) \ra U(n,m)/U(m).$
Then, since $\widehat{L}_{g^{-1}}:U(n,m)/U(m) \ra U(n,m)/U(m)$ is an isometry,
$\hat{\pi} \circ L_{g^{-1}}= \widehat{L}_{g^{-1}}\circ\hat{\pi}$
says that
$x$ is also horizontal with respect to $\hat{\pi}: U(n,m)\ra U(n,m)/U(m),$
in other words, $|x| =|\hat{\pi}_* x|.$
More precisely,
$$|x| = |L_{g^{-1} *} \,x| = |\hat{\pi}_* L_{g^{-1} *} \,x| = |\widehat{L}_{g^{-1} *} \hat{\pi}_* x| = |\hat{\pi}_* x|.$$
Furthermore,
$L_{g^{-1} *} \,x \perp \big(\mathfrak{u}(n)\!+\!\mathfrak{u}(m)\big)$
also says that
$L_{g^{-1} *} \,x \perp \mathfrak{u}(n)$
and
$L_{g^{-1} *} \,x \perp \mathfrak{u}(m)$
at the same time,
which means
$\hat{\pi}_* L_{g^{-1} *} \,x=  L_{g^{-1} *} \,x +\mathfrak{u}(m)$ 
is perpendicular to 
$
  \{V+\mathfrak{u}(m)\;|\;V \in \mathfrak{u}(n)\}
$ 
at the origin of
$U(n,m)/U(m)$ from $\mathfrak{u}(n) \perp \mathfrak{u}(m)$
and so horizontal with respect to 
$\pi:U(n,m)/U(m) \ra(D_{n,m}, \langle \cdot, \cdot \rangle_2)$
because 
$\{V+\mathfrak{u}(m)\;|\;V \in \mathfrak{u}(n)\}$ is the 
kernel of $\pi_*$ of the bundle 
$\pi:U(n,m)/U(m) \ra(D_{n,m}, \langle \cdot, \cdot \rangle_2)$
at the origin of $U(n,m)/U(m)$ from $[\mathfrak{u}(n),\mathfrak{u}(m)]=0.$
Then, 
since 
$\widehat{L}_{g}: U(n,m)/U(m) \ra U(n,m)/U(m)$ is an isometry preserving the fibers of the bundle $\pi: U(n,m)/U(m) \ra D_{n,m},$
$\widehat{L}_{g \,*} \hat{\pi}_* L_{g^{-1} *} \,x$
is also horizontal with repect to 
$\pi:U(n,m)/U(m) \ra(D_{n,m}, \langle \cdot, \cdot \rangle_2).$
And from
$$
\widehat{L}_{g \,*} \hat{\pi}_* L_{g^{-1} *} \,x 
= \hat{\pi}_* L_{g \,*}  L_{g^{-1} *} \,x
= \hat{\pi}_* \,x,
$$
we get that
$\hat{\pi}_* \,x$ is horizontal with repect to 
$\pi:U(n,m)/U(m) \ra(D_{n,m}, \langle \cdot, \cdot \rangle_2)$
and that
$$|\hat{\pi}_* x|= |\pi_* \hat{\pi}_* x|_2,$$
Therefore, $\tilde{\pi}=\pi \circ \hat{\pi}$ gives
$$
|\tilde{\pi}_* x|_1 = |x| =|\hat{\pi}_* x|
= |\pi_* \hat{\pi}_* x|_2 = |\tilde{\pi}_* x|_2.
$$
Thus, we will not distinguish one metric on $D_{n,m}$ from the other one.

\begin{lemma} \label{lemma}
Given a nontrivial $X \in M_{m \times n}(\bbc),$ 
the Lie subalgebra of $\mathfrak{u}(n,m),$ generated by
$$
  \widehat{X} =
  \left(
        \begin{array}{cccc}
          O_n & X^* \\
           X  & O_m \\
        \end{array}
  \right),
  \quad  
  \widehat{iX} =
  \left(
        \begin{array}{cccc}
          O_n & -iX^* \\
          iX  & O_m \\
        \end{array}
  \right),
$$
is
$$
  \emph{Span}_{\bbr} \{
                         \widetilde{V}_k, \widetilde{X}_k, \widetilde{iX}_k 
                         \: \big| \:  k=1,2, \cdots 
                     \},
$$
where
$$
   \widetilde{V}_k =  \left(
                            \begin{array}{cccc}
                                i(X^*X)^k & O  \\
                                   O      & -i (XX^*)^k  \\
                            \end{array}
                      \right),
$$
$$
   \widetilde{X}_k =  \left(
                            \begin{array}{cccc}
                                   O_n     &  (X^*X)^{k-1}X^*\\
                               X(X^*X)^{k-1}   &    O_m  \\
                            \end{array}
                      \right),
$$
and
$$
   \widetilde{iX}_k =  \left(
                             \begin{array}{cccc}
                                    O_n      &  -i(X^*X)^{k-1}X^*\\
                                iX(X^*X)^{k-1}   &    O_m  \\
                             \end{array}
                       \right).
$$
Furthermore, for $q= \emph{rk}X,$
\begin{align} \label{basis}
  \emph{Span}_{\bbr} \{ \widetilde{V}_k \: \big| \:  k=1, \cdots  \} 
  =\emph{Span}_{\bbr} \{ \widetilde{V}_k \: \big| \:  k=1, \cdots, q  \} , 
\end{align}
is, at most, a $q$-dimensional subalgebra of 
$\mathfrak{u}(n) + \mathfrak{u}(m).$ 
\end{lemma}
\begin{proof}
The first assertion can be given by direct calculations. \\
For the second one, consider a bundle
$$U(n) \times U(m) \ra U(n,m) \stackrel{\tilde{\pi}} \lra D_{n,m}.$$ 

For $X \in M_{m \times n}(\bbc),$ the following Lemma \ref{sing} make us consider three matrices $A \in U(n), B \in U(m)$ and 
$\Sigma \in M_{m \times n}(\bbc)$ such that $X=B \Sigma A^*.$ Then
$$ X^*X = A\Sigma^*\Sigma A^*, \quad XX^* = B\Sigma\Sigma^* B^*$$
and
$$
  \widetilde {V}_k
  =     
        \left(
              \begin{array}{cccc}
                A & O \\
                O & B \\
              \end{array}
        \right) 
        \left(
              \begin{array}{cccc}
                i\, (\Sigma^*\Sigma)^k  &     O \\
                          O             & -i\, (\Sigma\Sigma^*)^k \\
              \end{array}
        \right) 
        \left(
              \begin{array}{cccc}
                A^* & O \\
                O & B^* \\
              \end{array}
        \right),       
$$
where
$$
  \Sigma^*\Sigma 
  = \text{diag}[
                \sigma_1^2, \cdots, \sigma_q^2,
                \underbrace{0, \: \cdots, 0}_{(n-q) \text{ times}}
               ]
  \,\in \,M_{n \times n}
$$
and
$$
  \Sigma\Sigma^* 
  = \text{diag}[
                \sigma_1^2, \cdots, \sigma_q^2,
                \underbrace{0, \: \cdots, 0}_{(m-q) \text{ times}}
               ]
  \,\in \,M_{m \times m}
$$
for $\sigma_1 \ge \cdots \ge \sigma_q>0,$ the positive square roots of the decreasingly ordered nonzero eigenvalues of $WW^*,$ which are the same as the decreasingly ordered nonzero eigenvalues of $W^*W.$
Thus, the equation (\ref{basis}) is trivially obtained and so it is obvious that its dimension is less that or equal to $q.$
And it is a Lie algebra from $[\widetilde{V}_k, \widetilde{V}_j]=0$ 
for $k,j =1,2, \cdots.$
\end{proof}

\medskip
The following Lemma on \emph{Singular value decomposition} plays an important role in this paper.
\begin{lemma} \cite{HJ} \label{sing}
Let $W \in M_{m \times n}(\bbc)$ be given, put $a=\emph{min}\{m,n\},$ and suppose that $\emph{rk}W=q.$ \\
\indent
\emph{(i)} 
There are unitary matrices $A \in U(n)$ and $B \in U(m)$ and a square diagonal matrix 
$$
   \Sigma_a =  \left(
                             \begin{array}{cccc}
                                 \sigma_1    &        &    0     \\
                                             & \ddots &          \\
                                     0       &        & \sigma_a  \\
                             \end{array}
               \right)
$$
such that 
$
 \sigma_1 \geq \sigma_2 \geq \cdots \geq \sigma_q > 0 
 = \sigma_{q+1} = \cdots =\sigma_a
$
and $W=B \Sigma A^*,$ in which
\begin{displaymath}
  \Sigma= \left\{
                 \begin{array}{ll}
               \Sigma_a   &\text{ in case of } n=m, \\
               (\Sigma_a \:\: O)^T \in M_{m\times n}  &\text{ in case of } n<m, \\
               (\Sigma_a \:\: O)   \in M_{m\times n}  &\text{ in case of } n>m.   
                 \end{array} 
          \right.
\end{displaymath}
\indent
\emph{(ii)}
The parameters $\sigma_1, \cdots, \sigma_q$ are the positive square roots of the decreasingly ordered nonzero eigenvalues of $WW^*,$ which are the same as the decreasingly ordered nonzero eiganvalues of $W^*W.$
\end{lemma}

\medskip
  Recall the following proposition, which gives a sufficient condition to determine whether a given complex surface $S$ in $D_{n,m}$ related to $(X,g)$ is totally geodesic or not.
\begin{proposition} \cite{KN} \label{affine}
 Let $(G,H,\sigma)$ be a symmetric space and
 $\mathfrak{g} = \mathfrak{h} + \mathfrak{m}$
 the canonical decomposition. Then there is a natural one-to-one correspondence between the set of linear subspaces $\mathfrak{m}'$ of $\mathfrak{m}$ such that
$[[\mathfrak{m}', \mathfrak{m}'], \mathfrak{m}'] \subset \mathfrak{m}'$
and the set of complete totally geodesic submanifolds $M'$ through the origin $0$ of the affine symmetric space $M=G/H,$ the correspondence being given by 
$\mathfrak{m}' = T_0 (M').$
\end{proposition}

\bigskip

\section{
          Holonomy displacement in the bundle
          $U(n) \ra U(n,1)/U(n) \ra U(n,1)/(U(n) \times U(1))$
        } 
        \label{special2}

\bigskip
                
Even though the result in this section can be obtained in view of Corollary \ref{corollary}, we deal with this section in the way  which will be used in Section \ref{sec}.

\bigskip

Given $X \in M_{1 \times n}(\bbc) \cong \bbc^n,$ consider 
$$
  \text{Span}_{\bbr}\{\widehat{X}, \widehat{iX}\} =\mathfrak{m}' 
  \subset \mathfrak{m} \subset \mathfrak{u}(n,m).
$$
Then, for $\lambda = \sqrt{XX^*} =|\widehat{X}|$
and for   $W = \tfrac{1}{\lambda} X,$
\begin{align*}
  &\widehat{W} 
  = \tfrac{1}{\lambda} \widehat{X}
  =
  \left(
        \begin{array}{cccc}
          O_n & W^* \\
          W & 0 \\
        \end{array}
  \right), 
  \\
  &\widehat{iW} 
  = \tfrac{1}{\lambda} \widehat{iX}
  =
  \left(
        \begin{array}{cccc}
           O_n  & -iW^*\\
          iW  & 0 \\
        \end{array}
  \right),
  \\
  &\widehat{V} \!:=\!
  \left(
        \begin{array}{cccc}
           iW^*W & O\\
           O & -i \\
        \end{array}
  \right)\!,
\end{align*}
which generate a 3-dimensional Lie algebra $\mathfrak{\widehat{g}}$
with $\widehat{G}$ its Lie group such that 
\begin{align} \label{totally_geo}
  [\widehat{W},\widehat{iW}] = 2 \widehat{V}, \quad
  [\widehat{V},\widehat{W}] = -2 \widehat{iW}, \quad
  [\widehat{V},\widehat{iW}] = 2 \widehat{W}.
\end{align}
Since $WW^* = 1,$
Lemma \ref{sing} says that,
for 
$\Sigma= \left( 1 \: 0 \cdots 0\right) \in M_{1 \times n}(\bbc),$ 
there are $A \in U(n)$ and $\mu \in \bbr$ 
such that
$$
  W 
  = B \Sigma A^* 
   \in M_{1 \times n}(\bbc),
$$
where  $B = (e^{i \mu}) \in U(1),$ 
and then 
for
$$
\Omega  = \left(
                 \begin{tiny} 
                  \begin{array}{ccc|c}
                     & && 0 \\
                     &A && \vdots \\
                     & && 0 \\
                     \hline
                     0 & \cdots &0& e^{i\mu} 
                  \end{array}
                 \end{tiny} 
          \right),
$$
we get another expressios for $\widehat{W},\widehat{iW},\widehat{V}$
through $\text{Ad}_{\Omega}:\mathfrak{u}(n,m) \ra \mathfrak{u}(n,m),$
\begin{align} \label{iso_1_n}
  \widehat{W}&= \Omega  
                \left(
                       \begin{tiny} 
                       \begin{array}{cccc|c}
                          && &  & 1\\
                          && &  & 0\\
                          & &O_n & &\vdots\\
                          &&&& 0\\
                          \hline
                          1 &0 & \cdots &0& 0 \\
                       \end{array}
                       \end{tiny}
                \right)
                \Omega^{-1}
\end{align}
\begin{align} \label{iso_2_n}
  \widehat{iW}&= \Omega  
                \left(
                       \begin{tiny} 
                       \begin{array}{cccc|c}
                          && &  & -i\\
                          && &  & 0\\
                          & &O_n & &\vdots\\
                          &&&& 0\\
                          \hline
                          i &0 & \cdots &0& 0 \\
                       \end{array}
                       \end{tiny}
                \right)
                \Omega^{-1}
\end{align}
and
\begin{align} \label{iso_3_n}
  \widehat{V} &= \Omega
                \left(
                       \begin{tiny} 
                       \begin{array}{cccc|c}
                         i & 0      &\cdots      & 0 & 0\\
                         0& 0      &\cdots      & 0 & 0\\
                         \vdots &\vdots       &\ddots      & \vdots &0\\
                         0 &0 &\cdots & 0&0\\
                         \hline
                         0 & 0      & 0      &\cdots&-i \\
                       \end{array}
                       \end{tiny}
                \right)
                \Omega^{-1}.
\end{align}
In fact, $WW^*=1$ implies that the nonzero eigenvalue of $W^*W$ 
consists of a simple 1 from Lemma \ref{sing}. More concretely, note that 
$$
  \widehat{V} 
  =
    \left(
          \begin{array}{cccc}
             i W^*W  &  O\\
              O &  -iWW^* \\
          \end{array}
    \right)
  =
    \left(
          \begin{array}{cccc}
             i W^*W  &  O\\
              O &  -i \\
          \end{array}
    \right)
$$
and
$$
  W^*W = A \Sigma^* \Sigma A^* =
    A 
    \left(
          \begin{array}{cccc}
              1 &  O\\
              O &  O_{n-1} \\
          \end{array}
    \right)
    A^*.
$$

Consider a complex surface $S_0$ related to $(X,e)$ in $D_{n,1},$ which is totally geodesic from Equation (\ref{totally_geo}) and from Proposition \ref{affine}, where $e$ is the identity of $U(n,m).$
For a smooth, simple, closed, orientaion-preserving curve $c:[0,1] \ra S_0,$    assume that 
$\hat{c}:[0,1] \ra U(n,1)/U(1),$ one of its horizontal lifts,  is given.
Let $\text{Area}(c)$ denote  the area of the region on the surface $S_0$ surrounded by $c$ with respect to the area form $\omega_0=\omega_{(X,e)}$ 
related to $(X,e)$ on $S_0.$
Put $A(c)$ denote the  area with respect to the induced metric on $S_0$ from the metric   on $D_{n,m},$ \emph{i.e.,} the area with respect to the induced area form $\omega_1$ on $S_0$ from the metric   on $D_{n,m}.$ 
Note that $\omega_0 = \omega_1$ from Proposition \ref{prop}.
$$
  \text{
       \emph{Claim}) 
       There exists an element $\Psi \in \mathfrak{u}(n)$ such that 
       }
$$
$$  
  \hat{c}(1) = \hat{c}(0) \cdot e^{\Psi} \quad \text{and} \quad 
  \text{Tr}(\Psi) = 2i \text{Area}(c) = 2i A(c).
$$

\medskip
Consider a curve $z(t)=r(t)e^{i \theta(t)}$ in $\bbc$ and 
another one $\tilde{\gamma}(t)$ in 
$\tilde{S}=\{e^{\widehat{zW}}| z \in \bbc\} \subset U(n,1)$ such that
$\tilde{\gamma}(t)=e^{\widehat{z(t)W}}$ is a lifting of $c.$ 
Then for a horizontal lifting of $\hat{c},$ under the identification of $\mathfrak{u}(n)$ and $\mathfrak{u}(n) + \{0\},$
we can find a curve 
\begin{align*}
  \Psi(t) 
  &= i \phi (t) W^*W \\
  &=
    \left(
          \begin{array}{cccc}
              i\phi (t) W^*W  &  O\\
              O &  0 \\
          \end{array}
    \right)   \\
  &=  
     \Omega
     \left(
           \begin{tiny} 
           \begin{array}{cccc|c}
              i \phi(t) & 0      &\cdots      & 0 & 0\\
              0& 0      &\cdots      & 0 & 0\\
              \vdots &\vdots       &\ddots      & \vdots &0\\
              0 &0 &\cdots & 0&0\\
              \hline
              0 & 0      & 0      &\cdots& 0 \\
           \end{array}
           \end{tiny}
     \right)
     \Omega^{-1}
    \: \in \mathfrak{u}(n)
\end{align*} 
such that 
$$
  \hat{c}(t)= \big(\tilde{\gamma}(t) \,U(1) \big) \cdot e^{\Psi(t)} 
            = (\tilde{\gamma}(t) e^{\Psi(t)}) \;U(1). 
$$
Note that
from
$$
  \widehat{z(t)W} = \Omega  
                    \left(
                          \begin{tiny} 
                          \begin{array}{cccc|c}
                             && &  & r(t) e^{-i \theta(t)}\\
                             && &  & 0\\
                             & &O_n & &\vdots\\
                             &&&& 0\\
                             \hline
                             r(t) e^{i \theta(t)} &0 & \cdots &0& 0 \\
                          \end{array}
                          \end{tiny}
                   \right)
                   \Omega^{-1},
$$
we get
$$
 \tilde{\gamma}(t)= \Omega  
                    \left(
                          \begin{tiny} 
                          \begin{array}{cccc|c}
                          \cosh{(r(t))}&0&\cdots&0&e^{-i\theta(t)}\sinh{(r(t))}\\
                             0&& &  & 0\\
                             \vdots& &I_{n-1} & &\vdots\\
                             0&&&& 0\\
                             \hline
                           e^{i\theta(t) }\sinh{(r(t))}&0&\cdots&0&\cosh{(r(t))}                                      
                          \end{array}
                          \end{tiny}
                   \right)
                   \Omega^{-1}.
$$
Then for $\bar{c}(t)=\tilde{\gamma}(t) e^{\Psi(t)},$
$$
 \bar{c}(t)= \Omega  
             \left(
                  \begin{tiny} 
                  \begin{array}{cccc|c}
              e^{i\phi(t)}\cosh{(r(t))}&0&\cdots&0&e^{-i\theta(t)}\sinh{(r(t))}\\
                    0&& &  & 0\\
                    \vdots& &I_{n-1} & &\vdots\\
                    0&&&& 0\\
                    \hline
               e^{i (\theta(t)+\phi(t))}\sinh{(r(t))}&0 &\cdots&0&\cosh{(r(t))}
                  \end{array}
                  \end{tiny}
             \right)
             \Omega^{-1}
$$
and from
\begin{align*}
  &L_{{\bar{c}(t)^{-1}}_{*}} \dot{\bar{c}}(t) + {\mathfrak u}(1) \\
  &= e^{-\Psi (t)}\tilde{\gamma}(t)^{-1}
     \big(
          \tilde{\gamma}'(t) e^{\Psi(t)}
          + \tilde{\gamma}(t) e^{\Psi(t)}\Psi'(t)
     \big)
     + {\mathfrak u}(1)  \\
  &= \big(
          e^{-\Psi (t)}\tilde{\gamma}(t)^{-1} \tilde{\gamma}'(t)e^{\Psi (t)} 
          + \Psi'(t)
     \big)
     + {\mathfrak u}(1),
\end{align*}
$\hat{c}(t) = \bar{c}(t) U(1)$ is horizontal if and only if
the first $(n \times n)-$block of
$e^{-\Psi (t)}\tilde{\gamma}(t)^{-1} \tilde{\gamma}'(t)e^{\Psi (t)} + \Psi'(t)$
is a zero matrix, in other words, 
$$ 
  i \big(\phi'(t) - \theta'(t) \sinh^2(r(t))\big) = 0,
  \quad \emph{i.e.,} \quad
  \phi'(t) = \theta'(t) \sinh^2(r(t)).
$$
So, for  
$\Psi := \Psi(1)-\Psi(0) \in \text{Span}_{\bbr}\{i(X^*X)\} \in \mathfrak{u}(n)$
and for the region $D(\subset S)$ enclosed by the given orientation curve $c,$
Proposition \ref{prop} says that
\begin{align*}
  \text{Tr}\big(\Psi \big) 
  &= i (\phi(1) - \phi(0))
  \\
  &= i \: \int_{0}^{1} \theta'(t) \sinh^2(r(t)) dt \\
  &= 2 i \int_{[0,1]}\tfrac{1}{2}\sinh^2{(r(t))}\,\theta'(t)dt  \\
  &= 2 i \int_{c}\:\tfrac{1}{2}\sinh^2{r}\,d\theta  \\
  &= 2 i \int_{D}d\Big(\tfrac{1}{2}\sinh^2{r}\,d\theta\Big) \\
  &= 2 i \int_D \omega_0 \\
  &= 2 i \,\text{Area}(c) \\
  &= 2 i \, A(c).
\end{align*}
Since $\tilde{\gamma}(0)=\tilde{\gamma}(1)$
and $[\Psi(0), \Psi(1)]=O_n$ in $\mathfrak{u}(n),$ we get
$$
\bar{c}(0)^{-1}\bar{c}(1) 
= \big(\tilde{\gamma}(0) e^{\Psi(0)}\big)^{-1} 
  \big(\tilde{\gamma}(1) e^{\Psi(1)}\big)
= e^{-\Psi(0)}  e^{\Psi(1)}
= e^{\Psi}
$$
and so
$$
  \hat{c}(1) = \bar{c}(1)K =(\bar{c}(0) \, e^{\Psi})\,K 
  =\bar{c}(0)K\cdot e^{\Psi} = \hat{c}(0) \cdot e^{\Psi},
$$
\emph{i.e.}, the holonomy displacement is given by
the right action of $e^{\Psi} \in U(n).$

\bigskip

\section{Proof of Proposition \ref{prop}}  \label{sec_prop}

\medskip
\begin{proof}
For the part (i),
let 
$$
  W= \tfrac{1}{\big|\widehat{X}\big|} X,
  \quad 
  \widehat{W}= \widehat{\tfrac{1}{\big|\widehat{X}\big|} X} 
             = \tfrac{1}{\big|\widehat{X}\big|}\widehat{X},
  \quad \text{and} \quad
  \widehat{iW}= \widehat{\tfrac{1}{\big|\widehat{X}\big|} iX}
              = \tfrac{1}{\big|\widehat{X}\big|}\widehat{iX}
$$
and consider an ordered pair 
$\big( \widehat{W}, \widehat{iW}\big),$
which induces an oriented orthonormal basis of each tangent space of $S_0$ and an area form $\omega_0 = \omega_{(X,e)}$
related to $(X,e)$ on $S_0,$ where $e$ is the identity of $U(n,m).$

Assume $\text{rk}W = q, \; \text{min}\{n,m\}=a \text{ and } \text{Max}\{n,m\}=b.$ Then, Lemma \ref{sing} says that
there are nonnegative real numbers 
$$
  \sigma_1 \geq \cdots \geq \sigma_q > 0 = \sigma_{q+1} = \cdots = \sigma_a 
   = \cdots = \sigma_b = \cdots = \sigma_{n+m}
$$ 
such that
$\sigma_1, \cdots, \sigma_q$ are positive square roots of nonzero eigenvalues of $WW^*$ and $W^*W$ and that
$\Sigma_a = \text{diag}[\sigma_1, \cdots, \sigma_a]$ and
$$ W= B \Sigma A^*, \qquad A \in \text{U}(n), \; B \in \text{U}(m),   $$
where 
\begin{displaymath}
  \Sigma= \left\{
                 \begin{array}{ll}
               \Sigma_a   &\text{ in case of } n=m, \\
               (\Sigma_a \:\: O)^T \in M_{m\times n}  &\text{ in case of } n<m, \\
               (\Sigma_a \:\: O)   \in M_{m\times n}  &\text{ in case of } n>m.   
                 \end{array} 
          \right.
\end{displaymath}

For $r \geq 0$ and for $\theta \in \bbr,$ let $z = r e^{i \theta}$ and consider 
$\tilde{S}= \{e^{\widehat{z W}} \, | \, z \in \bbc  \}.$
Let
\begin{align*}
  \Omega
  =   
        \left(
              \begin{array}{cccc}
                A & O \\
                O & B \\
              \end{array}
        \right). 
\end{align*}
Then,
\begin{align*}
  e^{\widehat{zW}} 
  &=   
        \left(
              \begin{array}{cccc}
                \sum_{k=0}^{\infty}\frac{1}{(2k)!} r^{2k}(W^*W)^k 
                & e^{-i\theta}\sum_{k=0}^{\infty}\frac{1}{(2k+1)!} 
                    r^{2k+1}(W^*W)^kW^*\\
                e^{i\theta}\sum_{k=0}^{\infty}\frac{1}{(2k+1)!} 
                    r^{2k+1}W(W^*W)^k
                & \sum_{k=0}^{\infty}\frac{1}{(2k)!} r^{2k}(WW^*)^k \\
              \end{array}
        \right) 
  \\      
  &=   
        \left(
              \begin{array}{cccc}
                A & O \\
                O & B \\
              \end{array}
        \right) 
        \left(
              \begin{array}{cccc}
                \Gamma_n(z) & \Lambda(z)^* \\
                \Lambda(z)  & \Gamma_m(z) \\
              \end{array}
        \right) 
        \left(
              \begin{array}{cccc}
                A^* & O \\
                O & B^* \\
              \end{array}
        \right)       
  \\
  &=    \Omega
        \left(
              \begin{array}{cccc}
                \Gamma_n(z) & \Lambda(z)^* \\
                \Lambda(z)  & \Gamma_m(z) \\
              \end{array}
        \right) 
        \Omega^*
\end{align*}
from
$$
  W(W^*W)^k = (WW^*)^kW, \:
  W^*W=A\Sigma^*\Sigma A^* \: 
  \text{ and } 
  WW^*=B\Sigma\Sigma^*B^*,
$$
where
$$\Gamma_n(z) = \text{diag}[\cosh{(\sigma_1r)}, \cdots, \cosh{(\sigma_nr)}], $$
$$\Gamma_m(z) = \text{diag}[\cosh{(\sigma_1r)}, \cdots, \cosh{(\sigma_mr)}], $$
and for 
$
  \Lambda_a(z) = 
  \text{diag}[e^{i\theta}\sinh{(\sigma_1r)},\cdots,e^{i\theta}\sinh{(\sigma_ar)}],
$
\begin{displaymath}
  \Lambda(z)= 
     \left\{
            \begin{array}{ll}
              \Lambda_a(z)   &\text{ in case of } a=n=m, \\
              (\Lambda_a(z)\:\: O)^T \in M_{m\times n}&\text{ in case of }a=n<m,\\
              (\Lambda_a(z)\:\: O)   \in M_{m\times n}  
              &\text{ in case of }a=m<n.   
            \end{array} 
     \right.
\end{displaymath}
Note  $z \mapsto  e^{\widehat{zW}}: \, \bbc \ra \tilde{S}$ is a bijection.

By abusing notations, we use the same letter $(r,\theta)$ on $S_0$
for the induced one from the coordinate chart $(r,\theta)$ on 
$\tilde{S}.$
Then, the relation
\begin{align} \label{coord_vector}
  \tilde{\pi}_* \tfrac{\partial}{\partial r} 
      =  \tfrac{\partial}{\partial r} \circ \tilde{\pi}
  \quad \text{and} \quad
  \tilde{\pi}_*\tfrac{\partial}{\partial\theta} 
      = \tfrac{\partial}{\partial\theta}\circ\tilde{\pi}
\end{align}
holds. Note that,
under the identification of the Lie algebra ${\mathfrak u}(n,m)$ and 
$T_e G,$ where $G=U(n,m),$ the direct calculation shows that
\begin{align} \label{eqn_r} 
 {L_{e^{-\widehat{zW}}}}_{*} 
 \tfrac{\partial}{\partial r}\big|_{e^{\widehat{zW}}}  
 =
   \left( 
          \begin{array}{cccc}
                      O     & e^{-i\theta}W^* \\
              e^{i\theta}W  &      O  
          \end{array}
   \right),
\end{align}
and
\begin{align} \label{eqn_theta}
 &{L_{e^{-\widehat{zW}}}}_{*} 
 \tfrac{\partial}{\partial \theta}\big|_{e^{\widehat{zW}}} \\
 \nonumber
 &=
        \left(
              \begin{array}{cccc}
                \tfrac{-i}{2}\sum_{k=0}^{\infty}\frac{1}{(2k)!} (2r)^{2k}(W^*W)^k 
                & \tfrac{-i}{2} e^{-i\theta}\sum_{k=0}^{\infty}\frac{1}{(2k+1)!} 
                    (2r)^{2k+1}(W^*W)^kW^*\\
               \tfrac{i}{2} e^{i\theta}\sum_{k=0}^{\infty}\frac{1}{(2k+1)!} 
                    (2r)^{2k+1}W(W^*W)^k
                & \tfrac{i}{2}\sum_{k=0}^{\infty}\frac{1}{(2k)!}(2r)^{2k}(WW^*)^k\\
              \end{array}
        \right), 
\end{align}
so for $W^*W = A\Sigma^*\Sigma A^* = A\Sigma_n^*\Sigma_n A^*,$ 
where $\Sigma_n=\text{diag}[\sigma_1, \cdots, \sigma_n],$
we get
\begin{align*}
  &\omega_0
       \Big(
         \tfrac{\partial}{\partial r}|_{ \tilde{\pi}(e^{\widehat{zW}}) },\,
         \tfrac{\partial}{\partial\theta}|_{ \tilde{\pi}(e^{\widehat{zW}}) }
        \Big)  \\
  &= \text{det}
        \left(
              \begin{array}{cccc}
                \langle 
                   {L_{e^{-\widehat{zW}}}}_{*} 
                   \tfrac{\partial}{\partial r}\big|_{e^{\widehat{zW}}},\;  
                   \widehat{W}
                \rangle 
                & 
                \langle 
                   {L_{e^{-\widehat{zW}}}}_{*} 
                   \tfrac{\partial}{\partial \theta}
                   \big|_{e^{\widehat{zW}}},\;
                   \widehat{W}
                \rangle \\
                \langle 
                   {L_{e^{-\widehat{zW}}}}_{*} 
                   \tfrac{\partial}{\partial r}\big|_{e^{\widehat{zW}}},\;  
                   \widehat{iW}
                \rangle 
                & 
                \langle 
                   {L_{e^{-\widehat{zW}}}}_{*} 
                   \tfrac{\partial}{\partial \theta}
                   \big|_{e^{\widehat{zW}}},\;
                   \widehat{iW}
                \rangle \\
              \end{array}
        \right) \\
  &= \text{det}
        \left(
             \begin{array}{cccc}
               \cos{\theta} 
               & \frac{-1}{2} \sin{\theta} 
                 \sum_{k=0}^{\infty}\frac{1}{(2k+1)!} (2r)^{2k+1}
                 \text{Tr}(W^*W)^{k+1}\\
               \sin{\theta} 
               & \frac{1}{2} \cos{\theta} 
                 \sum_{k=0}^{\infty}\frac{1}{(2k+1)!} (2r)^{2k+1}
                 \text{Tr}(W^*W)^{k+1}
             \end{array}
        \right) \\
  &=  \frac{1}{2} 
      \sum_{k=0}^{\infty}\frac{1}{(2k+1)!} (2r)^{2k+1}
      (\sigma_1^{2(k+1)} + \cdots + \sigma_n^{2(k+1)})  \\
  &= \frac{1}{2} \sum_{j=1}^{n} \sigma_j \sinh{(2 \sigma_j r)}  \\
  &= \sum_{j=1}^{n} \sigma_j \sinh{(\sigma_j r)} \cosh{(\sigma_j r)}.
\end{align*}
Therefore,
$$
  \omega_0 \, = \, \sum_{j=1}^{n} \sigma_j  
                \cosh{(\sigma_j r)} \sinh{(\sigma_j r)} \,dr \wedge d\theta 
         \, = \, d\Big(
                       \sum_{j=1}^{n}  \tfrac{1}{2}\sinh^2{(\sigma_j r)} \,
                        d\theta
                  \Big).
$$
Furthermore, 
$$
   \widehat{W}^*\widehat{W}
   =
   \left( 
          \begin{array}{cccc}
             W^*W     & O \\
               O      & WW^*  
          \end{array}
   \right)
$$
shows that
$$
  1= |\widehat{W}|^2 = \text{Tr}(W^*W)= \text{Tr}(\Sigma_n^*\Sigma_n)
   = \sigma_1^2 + \cdots + \sigma_n^2. 
$$

\medskip
Before proving the part  (ii), note that the restriction of $\tilde{\pi}$ on $\tilde{S}$ is a bijection to $S_0$ from Remark \ref{bijection}, which induces that the restriction of $\mathbb{L}_{g}$ on $S_0$ is a bijection to $S$ from 
$\tilde{\pi} \circ L_g = \mathbb{L}_g \circ \tilde{\pi}.$
So, 
for $\mathbb{L}_{g^{-1}},$ its restriction
$\mathbb{L}_{g^{-1}}: S \ra S_0$ is also a bijection.
\\
\indent
To prove the part (ii), let $(r_1, \theta_1)$ be a coordinate chart on 
$S = \tilde{\pi} (g \tilde{S})$  with a complex structure induced from $\tilde{S},$ \emph{i.e.,} 
$$
r_1 \big( \tilde{\pi}(g e^{\widehat{zW}}) \big) 
= r\big(e^{\widehat{zW}}\big)
\quad \text{and} \quad
\theta_1 \big( \tilde{\pi}(g e^{\widehat{zW}}) \big) 
= \theta\big(e^{\widehat{zW}}\big),
$$
which says 
$$
 r_1 \circ \tilde{\pi} \circ L_g = r
 \quad \text{and} \quad
 \theta_1 \circ \tilde{\pi} \circ L_g = \theta.
$$
Then, from the equations (\ref{coord_vector}),
\begin{align*}
\tfrac{\partial}{\partial r_1}\big|_{\tilde{\pi}\circ L_g(e^{\widehat{zW}})}
&= \big(\tilde{\pi}\circ L_g\big)_*
   \tfrac{\partial}{\partial r}\big|_{e^{\widehat{zW}}}
\\
&= \big(\mathbb{L}_g \circ\tilde{\pi}\big)_*
   \tfrac{\partial}{\partial r}\big|_{e^{\widehat{zW}}}
\\
&= \mathbb{L}_{g *}
   \tfrac{\partial}{\partial r}\big|_{\tilde{\pi}(e^{\widehat{zW}})},
\end{align*}
so we get
$$
\tfrac{\partial}{\partial r}\big|_{\tilde{\pi}(e^{\widehat{zW}})}
=\mathbb{L}_{{g^{-1}} \,*}\,
  \tfrac{\partial}{\partial r_1}\big|_{\tilde{\pi}(g e^{\widehat{zW}})}.
$$
Similarly, we also get
$$
\tfrac{\partial}{\partial \theta}\big|_{\tilde{\pi}(e^{\widehat{zW}})}
=\mathbb{L}_{{g^{-1}} \,*}\,
  \tfrac{\partial}{\partial \theta_1}\big|_{\tilde{\pi}(g e^{\widehat{zW}})}.
$$
Then, 
\begin{align*}
  &\omega_{(X,g)}
       \Big(
         \tfrac{\partial}{\partial r_1}\big|_{ \tilde{\pi}(g e^{\widehat{zW}}) }
         ,\,
         \tfrac{\partial}{\partial\theta_1}\big|_{\tilde{\pi}(ge^{\widehat{zW}})}
        \Big) \\       
  &=\omega_{(X,g)}
       \Big(
           \tilde{\pi}_*  L_{g *}
            \tfrac{\partial}{\partial r}\big|_{e^{\widehat{zW}}}
           ,\,
           \tilde{\pi}_*  L_{g *}
            \tfrac{\partial}{\partial \theta}\big|_{e^{\widehat{zW}}}
        \Big) \\       
  &= \text{det}
        \left(
              \begin{array}{cccc}
                \langle 
                   {L_{(g e^{\widehat{zW}})^{-1}}}_{*} 
                    L_{g *}\tfrac{\partial}{\partial r}\big|_{e^{\widehat{zW}}}
                   ,\;  
                   \widehat{W}
                \rangle 
                & 
                \langle 
                   {L_{(g e^{\widehat{zW}})^{-1}}}_{*} 
                    L_{g *}\tfrac{\partial}{\partial \theta}
                    \big|_{e^{\widehat{zW}}}
                   ,\;
                   \widehat{W}
                \rangle \\
                \langle 
                   {L_{(g e^{\widehat{zW}})^{-1}}}_{*} 
                    L_{g *}\tfrac{\partial}{\partial r}\big|_{e^{\widehat{zW}}}
                    ,\;  
                   \widehat{iW}
                \rangle 
                & 
                \langle 
                   {L_{(g e^{\widehat{zW}})^{-1}}}_{*} 
                    L_{g *}\tfrac{\partial}{\partial \theta}
                   \big|_{e^{\widehat{zW}}}
                   ,\;
                   \widehat{iW}
                \rangle \\
              \end{array}
        \right) \\
  &= \text{det}
        \left(
              \begin{array}{cccc}
                \langle 
                   {L_{e^{-\widehat{zW}}}}_{*} 
                   \tfrac{\partial}{\partial r}\big|_{e^{\widehat{zW}}},\;  
                   \widehat{W}
                \rangle 
                & 
                \langle 
                   {L_{e^{-\widehat{zW}}}}_{*} 
                   \tfrac{\partial}{\partial \theta}
                   \big|_{e^{\widehat{zW}}},\;
                   \widehat{W}
                \rangle \\
                \langle 
                   {L_{e^{-\widehat{zW}}}}_{*} 
                   \tfrac{\partial}{\partial r}\big|_{e^{\widehat{zW}}},\;  
                   \widehat{iW}
                \rangle 
                & 
                \langle 
                   {L_{e^{-\widehat{zW}}}}_{*} 
                   \tfrac{\partial}{\partial \theta}
                   \big|_{e^{\widehat{zW}}},\;
                   \widehat{iW}
                \rangle \\
              \end{array}
        \right) \\
  &=\omega_0
       \Big(
         \tfrac{\partial}{\partial r}\big|_{ \tilde{\pi}(e^{\widehat{zW}}) },\,
         \tfrac{\partial}{\partial\theta}\big|_{ \tilde{\pi}(e^{\widehat{zW}}) }
        \Big)  \\
  &=\omega_0
       \Big(
           \mathbb{L}_{{g^{-1}} \,*}\,
           \tfrac{\partial}{\partial r_1}\big|_{\tilde{\pi}(g e^{\widehat{zW}})}       
            ,\,
           \mathbb{L}_{{g^{-1}} \,*}\,
           \tfrac{\partial}{\partial \theta_1}
           \big|_{\tilde{\pi}(g e^{\widehat{zW}})}       
        \Big)  \\
  &={\mathbb{L}_{g^{-1}}}^* \,\omega_0
       \Big(
           \tfrac{\partial}{\partial r_1}\big|_{\tilde{\pi}(g e^{\widehat{zW}})}       
            ,\,
           \tfrac{\partial}{\partial \theta_1}
           \big|_{\tilde{\pi}(g e^{\widehat{zW}})}       
        \Big).  \\
\end{align*}

\medskip

To prove part (iii), consider the induced area form $\omega$ on $S$  from the metric on $D_{n,m}.$ 
Since the restriction $\mathbb{L}_{g^{-1}}: S \ra S_0$ of 
$\mathbb{L}_{g^{-1}}: D_{n,m} \ra D_{n,m}$ is an isometry, we get
$$\omega =\mathbb{L}_{g^{-1}}^* \omega_1$$
for the induced area form $\omega_1$ on $S_0$  from the metric on $D_{n,m}.$ 
And, under the notation
$x^{\text{h}}$ of the horizontal part of a tangent vector 
$x \in T\big(U(n,m)\big)$ with respect to the Riemannian submersion $\tilde{\pi}:U(n,m) \ra D_{n,m},$ 
Equations (\ref{eqn_r}) and (\ref{eqn_theta})
show that 
$$
 {L_{e^{-\widehat{zW}}}}_{*} 
 \Big(
       \tfrac{\partial}{\partial r}\big|_{e^{\widehat{zW}}}  
 \Big)^{\text{h}}
 =
 \Big({L_{e^{-\widehat{zW}}}}_{*} 
       \tfrac{\partial}{\partial r}\big|_{e^{\widehat{zW}}}  
 \Big)^{\text{h}}
 \\
 =      {L_{e^{-\widehat{zW}}}}_{*} 
       \tfrac{\partial}{\partial r}\big|_{e^{\widehat{zW}}}  
$$
and that
\begin{align*}
 &{L_{e^{-\widehat{zW}}}}_{*} 
 \Big(
       \tfrac{\partial}{\partial \theta}\big|_{e^{\widehat{zW}}}  
 \Big)^{\text{h}}
 =
 \Big({L_{e^{-\widehat{zW}}}}_{*} 
       \tfrac{\partial}{\partial \theta}\big|_{e^{\widehat{zW}}}  
 \Big)^{\text{h}}
 \\
 &=
        \left(
              \begin{array}{cccc}
                O_n
                & \tfrac{-i}{2} e^{-i\theta}\sum_{k=0}^{\infty}
                  \frac{(2r)^{2k+1}}{(2k+1)!}  (W^*W)^kW^*\\
               \tfrac{i}{2} e^{i\theta}\sum_{k=0}^{\infty}
               \frac{(2r)^{2k+1}}{(2k+1)!}  W(W^*W)^k
                & O_m
              \end{array}
        \right)
 \\
 &=
        \left(
              \begin{array}{cccc}
                O_n
                & \tfrac{-i}{2} e^{-i\theta}\sum_{k=0}^{\infty}
                  \frac{(2r)^{2k+1}}{(2k+1)!}  A(\Sigma^*\Sigma)^k\Sigma^*B^*\\
               \tfrac{i}{2} e^{i\theta}\sum_{k=0}^{\infty}
               \frac{(2r)^{2k+1}}{(2k+1)!}  B\Sigma(\Sigma^*\Sigma)^k A^*
                & O_m
              \end{array}
        \right)
 \\
 &= 
   \Omega       
   \left( 
          \begin{array}{cccc}
                      O_n     & \tfrac{-i}{2} e^{-i\theta} \Upsilon(z)^* \\
              \tfrac{i}{2} e^{i\theta}\Upsilon(z)  &      O_m  
          \end{array}
   \right)
   \Omega^*,
\end{align*}
where
$$
  \Upsilon_a(z) = 
  \text{diag}  [\sinh{(2\sigma_1r)}\cdots,\sinh{(2\sigma_ar)}],
$$
and
\begin{displaymath}
  \Upsilon(z)= 
            \left\{
              \begin{array}{ll}
                \Upsilon_a   &\text{ in case of } n=m, \\
                (\Upsilon_a\:\: O)^T\in M_{m\times n}&\text{ in case of } n<m, \\
                (\Upsilon_a \:\: O)   \in M_{m\times n}  &\text{ in case of }n>m.   
              \end{array} 
            \right.
\end{displaymath}
Then,
\begin{align*}
  \Big|
       {L_{e^{-\widehat{zW}}}}_{*} 
       \tfrac{\partial}{\partial r}\big|_{e^{\widehat{zW}}}  
  \Big|
  &= \sqrt{
          \frac{1}{2}\text{Re}
          \Big(
             \text{Tr}
             \Big(
                 \big( 
                      {L_{e^{-\widehat{zW}}}}_{*} 
                      \tfrac{\partial}{\partial r}\big|_{e^{\widehat{zW}}} 
                 \big)^* 
                 {L_{e^{-\widehat{zW}}}}_{*} 
                 \tfrac{\partial}{\partial r}\big|_{e^{\widehat{zW}}}  
             \Big)
          \Big)
         }  
  \\
  &= \sqrt{\sigma_1^2 + \cdots + \sigma_n^2}
  \\
  &= 1,
\end{align*}
\begin{align*}
&2\,\langle
       {L_{e^{-\widehat{zW}}}}_{*} 
       \tfrac{\partial}{\partial r}\big|_{e^{\widehat{zW}}} 
       , 
       {L_{e^{-\widehat{zW}}}}_{*} 
       \tfrac{\partial}{\partial \theta}\big|_{e^{\widehat{zW}}}  
   \rangle
\\
&=\text{Re}
  \Big(
     \text{Tr}
     \Big(
          \big(
               {L_{e^{-\widehat{zW}}}}_{*} 
               \tfrac{\partial}{\partial r}\big|_{e^{\widehat{zW}}} 
          \big)^*    
          {L_{e^{-\widehat{zW}}}}_{*} 
          \tfrac{\partial}{\partial \theta}\big|_{e^{\widehat{zW}}}  
     \Big)
  \Big)
\\
&=\text{Re}
  \Big(
     \text{Tr}
     \Big(
         \Omega       
         \left( 
               \begin{array}{cccc}
                   O_n     &  e^{-i\theta}\Sigma^* \\
                   e^{i\theta} \Sigma  &      O_m  
               \end{array}
         \right)
         \left( 
               \begin{array}{cccc}
                         O_n     & \tfrac{-i}{2} e^{-i\theta} \Upsilon(z)^* \\
                 \tfrac{i}{2} e^{i\theta}\Upsilon(z)  &      O_m  
               \end{array}
         \right)
         \Omega^*
     \Big)
  \Big)
\\
&=\text{Re}
  \Big(
     \text{Tr}
     \Big(
         \left( 
               \begin{array}{cccc}
                      \tfrac{i}{2} \Sigma^* \Upsilon(z) & O  \\
                      O  &  \tfrac{-i}{2}\Sigma\Upsilon(z)^*  
               \end{array}
         \right)
     \Big)
  \Big)
\\
&=0
\end{align*}
and
\begin{align*}
 &\Big|
       \Big(
           {L_{e^{-\widehat{zW}}}}_{*} 
           \tfrac{\partial}{\partial \theta}\big|_{e^{\widehat{zW}}}  
       \Big)^{\text{h}}    
  \Big|
  \\
 &=\sqrt{
         \frac{1}{2}
         \text{Re}
         \Big( 
            \text{Tr}
            \Big(
                \Big(
                \big(
                    {L_{e^{-\widehat{zW}}}}_{*} 
                    \tfrac{\partial}{\partial \theta}\big|_{e^{\widehat{zW}}}  
                \big)^{\text{h}}    
                \Big)^*
                \big(
                    {L_{e^{-\widehat{zW}}}}_{*} 
                    \tfrac{\partial}{\partial \theta}\big|_{e^{\widehat{zW}}}  
                \big)^{\text{h}}    
            \Big)
         \Big) 
        }  
  \\
 &=\sqrt{
         \frac{1}{8} 
         \text{Re}
         \Big( 
            \text{Tr}
            \Big(
                 \left( 
                      \begin{array}{cccc}
                         \Upsilon(z)^*\Upsilon(z)   & O \\
                              O  &      \Upsilon(z)\Upsilon(z)^*  
                      \end{array}
                 \right)
            \Big)
         \Big) 
        }  
  \\
 &=\frac{1}{2}\sqrt{ \sum_{j=1}^n {\sinh^2(2\sigma_jr)} },
\end{align*}
which show the former one of the part (iii). 

To show the latter one, assume $\sigma_1 = \cdots = \sigma_q > 0 =\sigma_{q+1} = \cdots \sigma_n.$ Then,
$\sigma_1 = \tfrac{1}{\sqrt{q}}$ from $\sigma_1^2 + \cdots + \sigma_n^2 =1$
and
\begin{align*}
 \omega_1 
 &= \tfrac{\sqrt{q}}{2}\sinh{\big(\tfrac{2}{\sqrt{q}}r\big)}\, dr\wedge d\theta 
 \\
 &= \sqrt{q} \sinh{\big(\tfrac{1}{\sqrt{q}}r\big)} 
    \cosh{\big(\tfrac{1}{\sqrt{q}}r\big)} \, dr\wedge d\theta 
 \\
 &= d \Big(\tfrac{q}{2} \sinh^2{\big(\tfrac{1}{\sqrt{q}}r\big)} \Big) \, d\theta
 \\
 &= d\Big(\sum_{j=1}^{n}  \tfrac{1}{2}\sinh^2{(\sigma_j r)} \, d\theta   \Big)
 \\
 &= \omega_0.                  
\end{align*}
And 
$$
  W^*W = A\Sigma^* \Sigma A^*
  = A \,
    \text{diag}[
                \underbrace{\sigma_1^2, \cdots \sigma_1^2,}_{q \text{ times}}
                \underbrace{0, \cdots 0}_{(n-q) \text{ times}}
            \! ]  \,
    A^*           
$$
shows that
\begin{align*} 
  \text{det}(xI_n - W^*W) 
  &= \text{det}
    \big(
         A \,
        (
         xI_n- 
         \text{diag}[
                    \underbrace{\sigma_1^2, \cdots \sigma_1^2,}_{q \text{ times}}
                    \underbrace{0, \cdots 0}_{(n-q) \text{ times}}
                  \! ]  \,    
        )
         A^*
    \big) \\
  &= (x-\sigma_1^2)^q x^{(n-q)}. 
\end{align*}

\bigskip

\section{Proof of Theorem \ref{thm}}  \label{sec}

\medskip

We follow the notation in the proof of Proposition \ref{prop}.

\medskip

Let $K=U(m).$

\medskip
Note that the  equations
$$
\tilde{\pi} = \pi \circ \hat{\pi}, \quad
\hat{\pi} \circ L_g= \widehat{L}_g \circ \hat{\pi} \quad
\text{and} \quad
\tilde{\pi} \circ L_g= \mathbb{L}_g \circ \tilde{\pi}
$$
give
$$\pi \big(\widehat{L}_g(\hat{\pi}(\tilde{S}))\big) = S = \mathbb{L}_g (S_0).$$
And for any $g \in U(n,m)$ and for any $h_1 \in U(n),$ both 
$\widehat{L}_g, \widehat{R}_{h_1}: U(n,m)/U(m) \ra U(n,m)/U(m)$ are isometries preserving the  fibers of the bundle $\pi: U(n,m)/U(m) \ra D_{n,m}$
such that 
$$
\pi \circ \widehat{L}_g= \mathbb{L}_g \circ \pi 
\quad \text{and} \quad
\pi \circ \widehat{R}_{h_1} = \pi.
$$
Thus the curve $c$ on $S$ can be written by $\mathbb{L}_g \circ c_1$ for some curve $c_1$ on $S_0,$ and then for a horizontal lift $\hat{c}_1$ of $c_1$ on $U(n,m)/U(m),$  the horizontal lift $\hat{c}$ of $ c= \mathbb{L}_g \circ c_1$ will be given by 
$\widehat{R}_{h_1} \circ \widehat{L}_g \circ  \hat{c}_1$ for some 
$h_1 \in U(n).$ 
Note that if
$$ \hat{c}(1) = \hat{c}(0) \cdot h_0, \quad h_0 \in U(n)$$ then
$$
  (\widehat{R}_{h_1} \circ \widehat{L}_g \circ  \hat{c}_1)(1)
  = (\widehat{R}_{h_1} \circ \widehat{L}_g \circ  \hat{c}_1)(0) 
    \cdot \big( h^{-1}_1 \; h_0 \; h_1  \big).
$$
And, if $h_0 = e^{\Phi}$ for some $\Phi \in \mathfrak{u}(n),$ then
$$h^{-1}_1 \; h_0 \; h_1 = e^{\text{Ad}_{h^{-1}_1}\Phi}$$ 
and
$$\text{Tr}(\text{Ad}_{h^{-1}_1}\Phi) = \text{Tr}(\Phi).$$
So, Proposition \ref{prop}(ii) enables us to suppose that $g$ is the identity of $U(n,m)$ without loss of generality.

\bigskip
Before constructing the horizontal lift $\hat{c}(t)$ of 
$c(t) \in S_0, t \in [0,1],$ 
in the bundle
$ U(n,m)/U(m) \stackrel{\pi}\ra D_{n,m}, $
let $z(t) = r(t) e^{i \theta(t)}$ and 
$ \tilde{\gamma}(t) = e^{\widehat{z(t)W}},$ 
 the lift of $c$ in the bundle $U(n,m) \ra D_{n,m},$ lying in $\tilde{S} \subset U(n,m).$ 
From Lemma \ref{lemma}, consider the Lie subgroup 
$\widehat{G} \; \big(\!\subset \!U(n,m)\!\big)$ of the Lie algebra 
$\hat{\mathfrak{g}}$ generated by $\{\widehat{W}, \widehat{iW}\},$ which is the same as the one by $\{\widehat{X}, \widehat{iX}\}.$
Follow the notation of Lemma \ref{lemma} and let $\widehat{H}$ be the Lie group of the Lie algebra $\hat{\mathfrak{h}}$ generated by 
$\{\widetilde{V}_k | k=1,2, \cdots \}.$ Then $\widehat{H}$ is a subset of the intersection of $\widehat{G}$ and $U(n) \times U(m),$
which means that for each element of $\widehat{H},$ 
the map, given by its right multiplication in $\widehat{G}$ with the left invariant metric induced from the metric of $U(n,m)$ determined by (\ref{metric}), is an isometry of $\widehat{G}.$
Consider the inclusion map 
$\iota \!\!: \!\widehat{G} \hookrightarrow U(n,m),$
its Lie algebra homomorphism $d\iota: \hat{\mathfrak{g}} \ra \mathfrak{u}(n,m)$ 
and the Riemannian submersion 
$\widehat{pr}: \widehat{G} \ra \widehat{G}/\widehat{H}.$
Since 
\begin{align} \label{inclusion}
  d\iota(\widetilde{V}_k) = \widetilde{V}_k, \quad
  d\iota(\widetilde{X}_k) = \widetilde{X}_k \quad
  \text{and} \quad
  d\iota(\widetilde{iX}_k) = \widetilde{iX}_k
\end{align}
for $k=1,2, \cdots,$
Lemma \ref{lemma},  
the equation (\ref{eqn_r}) on
${L_{e^{-\widehat{zW}}}}_{*}\tfrac{\partial}{\partial r}\big|_{e^{\widehat{zW}}}$
and
the equation (\ref{eqn_theta}) on
$
 {L_{e^{-\widehat{zW}}}}_{*}
  \tfrac{\partial}{\partial \theta}\big|_{e^{\widehat{zW}}}
$
in the proof of Proposition \ref{prop}
show that $\tilde{S} $ is a subset of $\widehat{G}.$
The tangent vector of any horizontal curve $\tilde{c}$ in 
$\widehat{G} \ra \widehat{G}/\widehat{H}$
is spanned by $\widetilde{X}_{k}$ and by $\widetilde{iX}_k,$ $k=1,2, \cdots.$ 
The tangent vector of its induced curve $\iota \circ \tilde{c}$ in $U(n,m)$ 
is too, from the equation (\ref{inclusion}), which means that, 
by abusing of notation,
$\iota \circ \tilde{c} = \tilde{c}$ is a horizontal curve 
in the bundle $U(n,m) \ra D_{n,m}$
since both $\widetilde{X}_k$ and $\widetilde{iX}_k$ are still horizontal 
in this bundle.
In fact, $\widetilde{V}_k, \widetilde{X}_k$ and $\widetilde{iX}_k$ in $\mathfrak{g}$ are different from 
$\widetilde{V}_k, \widetilde{X}_k$ and $\widetilde{iX}_k$ in $\mathfrak{\hat{g}}$ as tangent vector fields, but they are
$\iota$-related, \emph{i.e.,}
\begin{align*} 
  \iota_* \widetilde{V}_k = \widetilde{V}_k \circ \iota, \quad
  \iota_* \widetilde{X}_k = \widetilde{X}_k \circ \iota \quad
  \text{and} \quad
  \iota_* \widetilde{iX}_k = \widetilde{iX}_k \circ \iota,
\end{align*}
and so $\dot{\overbrace{\iota \circ \tilde{c}}} = \iota_* \dot{\tilde{c}}$
is generated by $\widetilde{X}_k, \widetilde{iX}_k$ in $\mathfrak{g}.$
And the horizontal curve $\tilde{c}$ in $\widehat{G}$ with 
$\widehat{pr} \circ \tilde{c} = \widehat{pr} \circ \tilde{\gamma}$
can be given by
$\tilde{c}(t) = \tilde{\gamma}(t) e^{\widehat{\Psi}(t)}$
for some curve 
$\widehat{\Psi}: [0,1] \ra \hat{\mathfrak{h}} $
and will be regarded as a horizontal lift of $c$ 
in the bundle $U(n,m) \ra D_{n,m}.$ 
For 
\begin{align*}
  \widehat{\Psi}(t) 
  =   
        \left(
              \begin{array}{cccc}
                \eta_1(t) & O \\
                     O    & \eta_2(t) \\
              \end{array}
        \right)
   \in \hat{\mathfrak{h}} \subset \mathfrak{u}(n) + \mathfrak{u}(m),            
\end{align*}
let 
\begin{align*}
  \Psi(t) 
  =   
        \left(
              \begin{array}{cccc}
                \eta_1(t) & O \\
                     O    & O_m \\
              \end{array}
        \right)
   \in \mathfrak{u}(n) + \{O_m\}.            
\end{align*}
Then, for the curve $\bar{c}(t) = \tilde{\gamma}(t) e^{\Psi(t)}$ in $U(n,m),$
the equation
$$
  \tilde{c}(t)K 
  = \tilde{\gamma}(t) e^{\widehat{\Psi}(t)}K
  = \tilde{\gamma}(t) e^{\Psi(t)}K
  = \bar{c}(t)K
$$
holds. Therefore, 
to construct the horizontal lift $\hat{c}(t)=\tilde{c}(t)K = \bar{c}(t)K$ of $c(t)$ in the bundle
$ U(n,m)/U(m) \stackrel{\pi}\ra D_{n,m}, $
it suffices to find such a curve $\Psi (t) \in \mathfrak{u}(n).$

\bigskip
Note that
\begin{align}\label{decomp_gamma}
  \tilde{\gamma}(t) 
  &=   
        \left(
              \begin{array}{cccc}
                A & O \\
                O & B \\
              \end{array}
        \right) 
        \left(
              \begin{array}{cccc}
                \Gamma_n(t)  & \Lambda(t)^* \\
                \Lambda(t) & \Gamma_m(t) \\
              \end{array}
        \right) 
        \left(
              \begin{array}{cccc}
                A^* & O \\
                O & B^* \\
              \end{array}
        \right)
  \\             
  &=   
        \Omega
        \left(
              \begin{array}{cccc}
                \Gamma_n(t)  & \Lambda(t)^* \\
                \Lambda(t) & \Gamma_m(t) \\
              \end{array}
        \right) 
        \Omega^*,       
\end{align}
where
$$
  \Gamma_j(t) = \text{diag}[\cosh{(\sigma_1r(t))}, \cdots, \cosh{(\sigma_jr(t))}],
  \quad j=n,m  
$$
and for 
$
  \Lambda_a(t) = 
\text{diag}[e^{i\theta(t)}\sinh{(\sigma_1r(t))},  \cdots,e^{i\theta(t)}\sinh{(\sigma_ar(t))}],
$
\begin{displaymath}
  \Lambda(t)= 
     \left\{
            \begin{array}{ll}
              \Lambda_a(t)   &\text{ in case of } a=n=m, \\
              (\Lambda_a(t)\:\: O)^T \in M_{m\times n}&\text{ in case of }a=n<m,\\
              (\Lambda_a(t)\:\: O)   \in M_{m\times n}  
              &\text{ in case of }a=m<n.   
            \end{array} 
     \right.
\end{displaymath}
Lemma \ref{lemma}, $W^*W = \tfrac{1}{\big|\widehat{X}\big|^2} X^*X$
and $WW^* = \tfrac{1}{\big|\widehat{X}\big|^2} XX^*$
enable us to consider a curve 
$$
  \widehat{\Psi}(t) =   
        \left(
              \begin{array}{cccc}
                i \sum_{k=1}^{q} \phi_{k}(t)(W^*W)^{k} & O \\
                O  &  -i \sum_{k=1}^{q} \phi_{k}(t)(WW^*)^{k} \\
              \end{array}
        \right)
  \in  \hat{\mathfrak{h}} 
$$
such that $q=\text{rk}W$ and that
$$\tilde{c}(t)=\tilde{\gamma}(t) e^{\widehat{\Psi}(t)}.$$
Let
$$
  \Psi(t) =   
        \left(
              \begin{array}{cccc}
                i \sum_{k=1}^{q} \phi_{k}(t)(W^*W)^{k} & O \\
                O  &  O_{m} \\
              \end{array}
        \right)
  \in \mathfrak{u}(n).
$$
Then, by abusing of notations,we can say that
$\Psi(t) \in \text{Span}_{\bbr}\{i(X^*X)^k\}_{k=1}^{q}. $
And 
\begin{align*}
  \Psi(t)
  &=
  \Omega \:
  \text{diag}\Big[       
                   i \sum_{k=1}^{q} \sigma_1^{2k}\phi_k(t), 
                   \: \cdots, \: 
                   i \sum_{k=1}^{q} \sigma_{n}^{2k}\phi_k(t),
                   \underbrace{0, \: \cdots, \: 0}_{m\text{-times}}
              \Big] \:
  \Omega^*.  
\end{align*}
So, 
we get
\begin{align} \label{decomp_e_Psi}
  e^{\Psi (t)}
  &=
  \Omega \:
  \text{diag} \Big[
                    e^{i \sum_{k=1}^{q} \sigma_1^{2k}\phi_k(t)}, 
                    \:\cdots,\:                  
                    e^{i \sum_{k=1}^{q} \sigma_{n}^{2k}\phi_k(t)},
                    \underbrace{1, \:\cdots, \:  1}_{m\text{-times}}
              \Big] \:
   \Omega^*           
\end{align}
and
$$
  \hat{c}(t)
  =\tilde{\gamma}(t) e^{\widehat{\Psi}(t)} K
  =\tilde{\gamma}(t) e^{\Psi(t)} K
  =\bar{c}(t)K
$$
for $\bar{c}(t)=\tilde{\gamma}(t) e^{\Psi(t)}.$
Note 
\begin{align*}
  &L_{{\bar{c}(t)^{-1}}_{*}} \dot{\bar{c}}(t) + {\mathfrak u}(m) \\
  &= e^{-\Psi (t)}\tilde{\gamma}(t)^{-1}
     \big(
          \tilde{\gamma}'(t) e^{\Psi(t)}
          + \tilde{\gamma}(t) e^{\Psi(t)}\Psi'(t)
     \big)
     + {\mathfrak u}(m)  \\
  &= \big(
          e^{-\Psi (t)}\tilde{\gamma}(t)^{-1} \tilde{\gamma}'(t)e^{\Psi (t)} 
          + \Psi'(t)
     \big)
     + {\mathfrak u}(m),
\end{align*}
and the direct calculation 
$e^{-\Psi (t)}\tilde{\gamma}(t)^{-1} \tilde{\gamma}'(t)e^{\Psi (t)} + \Psi'(t)$
through the singular value decompositions (\ref{decomp_gamma}) and (\ref{decomp_e_Psi})  of $\tilde{\gamma}$ and $e^{\Psi(t)}$ says that 
its first $(n \times n)$-block matrix is
$$
  A  
  \:\:
  \text{diag} \Big[ 
                    i \,
                    \Big( 
                          \sum_{k=1}^{q} \sigma_j^{2k}\phi_k'(t)
                          -\theta'(t) \sinh^2(\sigma_jr(t))
                    \Big) 
              \Big]_{j=1}^n 
  \:\:
  A^*.
$$
Note that $\sigma_{q+1} = \cdots =\sigma_{n} = 0$  if $n\! > \!q= \text{rk}W.$
\\
Then
$\hat{c}(t)=\bar{c}(t)K$ is a horizontal curve in $G/K$ if and only if 
$$
  \sum_{k=1}^{q}\sigma_j^{2k}\phi_k'(t)
  = \theta'(t) \sinh^2(\sigma_jr(t)),
  \qquad j=1, \cdots, q.
$$
Therefore, Proposition \ref{prop} says that, for 
$
  \Psi:= \Psi(1) - \Psi(0) 
  \in \text{Span}_{\bbr}\{i(X^*X)^k\}_{k=1}^{q} \subset \mathfrak{u}(n)
$
and for the region $D (\subset S_0)$ enclosed by the given orientation-preserving curve $c:[0,1] \ra S_0,$ 
\begin{align*}
  \text{Tr}\big(\Psi \big) 
  &= i \: \sum_{j=1}^n\sum_{k=1}^{q}\sigma_j^{2k}(\phi_k(1) - \phi_k(0))
  \\
  &= i \: \int_{0}^{1} \sum_{j=1}^{n}  \theta'(t) \sinh^2(\sigma_jr(t)) dt \\
  &= 2 i \int_{[0,1]}\:\sum_{j=1}^{n}\tfrac{1}{2}\sinh^2{(\sigma_jr(t))}\,\theta'(t)dt  \\
  &= 2 i \int_{c}\:\sum_{j=1}^{n}\tfrac{1}{2}\sinh^2{(\sigma_jr)}\,d\theta  \\
  &= 2 i \int_{D}d\Big(\sum_{j=1}^{n}\tfrac{1}{2}\sinh^2{(\sigma_jr)}\,d\theta\Big) \\
  &= 2 i \int_D \omega_0 \\
  &= 2 i \,\text{Area}(c).
\end{align*}
And, $\tilde{\gamma}(0) = \tilde{\gamma}(1)$ and 
the Lie bracket $[i(X^*X)^k, i(X^*X)^j]=O_n$ in 
$\mathfrak{u}(n)$ say that 
$$
\bar{c}(0)^{-1}\bar{c}(1) 
= \big(\tilde{\gamma}(0) e^{\Psi(0)}\big)^{-1} 
  \big(\tilde{\gamma}(1) e^{\Psi(1)}\big)
= e^{-\Psi(0)}  e^{\Psi(1)}
= e^{\Psi}.
$$
Since the holonomy displacement is given by
the right action of $e^{\Psi} \in U(n),$ \emph{i.e.},
$$
  \hat{c}(1) = \bar{c}(1)K =(\bar{c}(0) \, e^{\Psi})\,K 
  =\bar{c}(0)K\cdot e^{\Psi} = \hat{c}(0) \cdot e^{\Psi},
$$
the theorem is proved.

\end{proof}

\bibliographystyle{amsplain}

\end{document}